\documentclass[a4paper]{article}
\usepackage{amsmath,amsthm,amssymb}

\usepackage{color}
\usepackage[dvips]{graphicx} 
\usepackage{version}
\newenvironment{NB}{
\color{red}{\bf NB}. \footnotesize
}{}
\excludeversion{NB}

\usepackage[all]{xy}

\newtheorem{thm}{Theorem}[section]
\newtheorem{defn}[thm]{Definition}
\newtheorem{ex}[thm]{Example}
\newtheorem{prop}[thm]{Proposition}
\newtheorem{cor}[thm]{Corollary}
\newtheorem{lem}[thm]{Lemma}

\newtheorem{rem}[thm]{Remark}
\newtheorem{ques}[thm]{Question}

\newcommand{\mf}[1]{{\mathfrak{#1}}}
\newcommand{\mr}[1]{{\mathrm{#1}}}
\newcommand{\mb}[1]{{\mathbf{#1}}}
\newcommand{\bb}[1]{{\mathbb{#1}}}
\newcommand{\mca}[1]{{\mathcal{#1}}}

\newcommand{\Hom}{\mr{Hom}}

\newcommand{\dimv}{\underline{\dim}}

\newcommand{\Z}{\bb{Z}}
\newcommand{\Zh}{\bb{Z}_{\mathrm{h}}}
\newcommand{\C}{\bb{C}}
\newcommand{\CP}{\bb{P}}

\newcommand{\R}{\bb{R}}

\newcommand{\A}{\mca{A}}
\newcommand{\F}{\mca{F}}
\newcommand{\mcaS}{\mca{S}}

\newcommand{\vv}{\mb{v}}

\newcommand{\Ih}{I_{\mathrm{h}}}

\newcommand{\perv}{{}^{-1}\mr{Per}(Y/X)}

\newcommand{\cohy}{\mr{Coh}(Y_\sigma)}
\newcommand{\cohcy}{\mr{Coh}_{\mathrm{cpt}}(Y_\sigma)}
\newcommand{\dcohy}{D^b\mr{Coh}(Y_\sigma)}
\newcommand{\dcohcy}{D^b_{\mathrm{cpt}}\mr{Coh}(Y_\sigma)}

\newcommand{\moda}{\mr{mod}A_\sigma}
\newcommand{\modfa}{\mr{mod}_{\mathrm{fin}}A_\sigma}
\newcommand{\dmoda}{D^b(\mr{mod}A_\sigma)}
\newcommand{\dmodfa}{D^b_{\mathrm{fin}}(\mr{mod}A_\sigma)}

\newcommand{\D}[2]{\mathcal{D}^{\zeta^\circ}_{\mathrm{fin}}[#1,#2)}
\newcommand{\Dinf}[2]{\mathcal{D}^{\zeta^\circ}[#1,#2)}

\newcommand{\catAzeta}{\mathcal{A}_{\mathrm{fin}}^\zeta}
\newcommand{\catAinfzeta}{\mathcal{A}^\zeta}

\newcommand{\MM}{\mf{M}(\sss\underline{\nu},\underline{\lambda},\zeta;\vv)}
\newcommand{\MMM}{\mf{M}^{\mathrm{ncDT}}(\zeta,\underline{\nu},\underline{\lambda}\,;\vv)}

\newcommand{\V}{\mca{V}}

\newcommand{\unu}{\underline{\nu}}
\newcommand{\ulam}{\underline{\lambda}}
\newcommand{\uemp}{\underline{\emptyset}}

\newcommand{\I}{\mca{I}}
\newcommand{\II}{\mca{I}_{\underline{\nu},\underline{\lambda}}}
\newcommand{\III}{P^{\zeta}_{\underline{\nu},\underline{\lambda}}}

\newcommand{\PP}{P^\zeta_{\underline{\nu},\underline{\lambda}}}

\newcommand{\zc}{\zeta^\circ}

\newcommand{\Apair}{\sigma\sss;\sss\underline{\nu},\underline{\lambda}}
\newcommand{\ztpair}{\sigma,\zeta\sss;\sss\underline{\nu},\underline{\lambda}}

\newcommand{\K}{\mathbb{K}}
\newcommand{\fock}{({\wedge^{\hspace{-2pt}\frac{\infty}{2}}})_0}
\newcommand{\pg}{\overset{+}{\succ}}

\newcommand{\mg}{\overset{-}{\succ}}

\newcommand{\tenchi}{{}^{\mathrm{t}}}

\newcommand{\Vmin}{\V_{\mathrm{min}}}

\newcommand{\type}{\sigma,\theta\sss;\sss\underline{\nu}, \underline{\lambda}}

\newcommand{\surj}{\twoheadrightarrow}

\newcommand{\sss}{\hspace{0.5pt}}

\usepackage{pst-all}

\numberwithin{equation}{section}

\title{Non-commutative Donaldson-Thomas theory and vertex operators}
\author{Kentaro Nagao\\
RIMS, Kyoto University\\
  Kyoto 606-8502, Japan
}

\begin{document}

\maketitle
\begin{abstract}
In \cite{open_3tcy}, we introduced a variant of non-commutative Donaldson-Thomas theory in a combinatorial way, which is related to the topological vertex by a wall-crossing phenomenon. 
In this paper, we (1) provide an alternative definition in a geometric way, (2) show that the two definitions agree with each other and (3) compute the invariants using the vertex operator method, following \cite{ORV} and \cite{young-mckay}. 
The stability parameter in the geometric definition determines the order of the vertex operators and 
hence we can understand the wall-crossing formula in non-commutative Donaldson-Thomas theory as the commutator relation of the vertex operators.
\end{abstract}

\section*{Introduction}
Let $X:=\{x_1x_2={x_3}^{L^+}{x_4}^{L^-}\}$ be a affine toric Calabi-Yau $3$-fold, which corresponds to the trapezoid with height 1, with length $L_+$ edge at the top and $L_-$ at the bottom. Let $Y_\sigma \to X$ be a crepant resolution of $X$.
Note that $Y_\sigma$ has $L+2$ affine lines as torus invariant closed subvarieties ($L:=L_++L_-$).
In other words, there are $L+2$ open edges in the toric graph of $Y_\sigma$.
Given an $(L+2)$-tuple of Young diagrams
\[
(\sss\underline{\nu},\underline{\lambda}\sss)=\Bigl(\nu_+,\nu_-,\lambda^{(1/2)},\ldots,\lambda^{(L-1/2)}\sss \Bigr)\quad (L:=L^++L^-)
\]
associated with $L+2$ open edges (see Figure \ref{fig_toric_graph}), we can define a torus invariant ideal sheaf $\II$ on $Y_\sigma$ (\S \ref{subsec_21}) and a moduli space $\mathfrak{M}^{\mathrm{DT}}(\sss\underline{\nu},\underline{\lambda}\sss)$ of quotient sheaves of $\II$
(\S \ref{subsec_openDT}).
Note that $\mathcal{I}_{\underline{\emptyset},\underline{\emptyset}}=\mathcal{O}_{Y_\sigma}$ and hence the moduli space $\mathfrak{M}^{\mathrm{DT}}(\sss\underline{\emptyset},\underline{\emptyset}\sss)$ is the Hilbert scheme of closed subschemes of $Y_\sigma$.
We define Euler characteristic version of {\it open} (commutative) {\it Donaldson-Thomas invariants} by the Euler characteristics of the connected components of $\mathfrak{M}^{\mathrm{DT}}(\sss\underline{\nu},\underline{\lambda}\sss)$\footnote{The word ``open'' stems from such terminologies as ``open topological string theory''. According to \cite{TV}, open topological string partition function is given by summing up the generating functions of these invariants over Young diagrams.}.
The torus action of $Y_\sigma$ induces the torus action on $\mathfrak{M}^{\mathrm{DT}}(\sss\underline{\nu},\underline{\lambda}\sss)$.
The torus fixed point set is isolated and parametrized in terms of $L$-tuples of $3$-dimensional Young diagrams.
Thus the generating function of the open Donaldson-Thomas invariants 
can be described in terms of topological vertex (\cite{TV,mnop}, see \S \ref{subsec_tv_via_vo}).

Let $A_\sigma$ be a non-commutative crepant resolution of the affine toric Calabi-Yau $3$-fold $X$.
We can identify the derived category of coherent sheaves on $Y_\sigma$ and the one of $A_\sigma$-modules by a derived equivalence.
A parameter $\zeta$ gives a Bridgeland's stability condition of this derived category, and hence a core $\catAinfzeta$ of a t-structure  on it (Definition \ref{defn_catA}).
In fact, we have two specific parameters such that the corresponding t-structures coincide with the ones given by $Y_\sigma$ or $A_\sigma$ respectively. 
Given an element in  $\A^\zeta$, we can restrict it to get a sheaf on the smooth locus $X^{\mathrm{sm}}$. 
Since the singular locus $X^{\mathrm{sing}}$ is compact, 
it makes sense to study those elements in $\A^\zeta$ which are isomorphic to $\II$ outside a compact subset of $X$, or in other words, those elements in $\A^\zeta$ which have the same asymptotic behavior as $\II$. 
We will study the moduli spaces of such objects as noncommutative analogues of $\mathfrak{M}^{\mathrm{DT}}(\sss\underline{\nu},\underline{\lambda}\sss)$. 
\begin{NB}
We want to study moduli spaces 
\[
\bigl\{
\mathcal{I}_{(\sss\underline{\nu},\underline{\lambda}\sss)}\to V\mid V\in \mathrm{mod}_{\mathrm{fin}}\hspace{-1pt}A_\sigma
\bigr\}
\]
of pairs of finite dimensional $A_\sigma$-modules $V$ and morphisms $\mathcal{I}_{(\underline{\nu},\underline{\lambda})}\to V$ in the derived category
\footnote{Let $H^*_{\moda}(-)$ denote the cohomology with respect to the t-structure corresponding to $\moda$. Since the derived equivalence is given by a tilting, we have $H^i_{\moda}(\mathcal{I}_{(\underline{\nu},\underline{\lambda})})=0$ unless $i=0,-1$. In particular, giving a morphism $\mathcal{I}_{(\underline{\nu},\underline{\lambda})}\to V$ is equivalent to giving a morphism $H^0_{\moda}(\mathcal{I}_{(\underline{\nu},\underline{\lambda})})\to V$.}
.
\end{NB}%
In general the ideal sheaf $\II$ is not an element in $\A^\zeta$, 
however $\PP:=H^0_{\A^\zeta}(\II)$ is always in $\A^\zeta$.
We will construct the moduli space $\mathfrak{M}^{\mathrm{ncDT}}(\zeta,\underline{\nu},\underline{\lambda})$ of quotients of $\PP$ in $\A^\zeta$
as a GIT quotient (\S \ref{subsec_proof}).
Note that $\mathfrak{M}^{\mathrm{ncDT}}(\zeta,\underline{\emptyset},\underline{\emptyset})$ is the moduli space we have studied in \cite{3tcy}. 
We define Euler characteristic version of {\it open non-commutative Donaldson-Thomas invariants} by the Euler characteristics of the connected components of $\mathfrak{M}^{\mathrm{ncDT}}(\zeta,\underline{\nu},\underline{\lambda})$\footnote{The reader may also refer to \cite{nagao-yamazaki} in which we study the invariants in the physics context.}.

\begin{NB}
In order to compute the invariants, we will give another description of the moduli space. 
Given $\zeta$, we can construct a core $\catAzeta$ of a t-structure of the derived category and an element $\III$ in $\catAzeta$.
We define
\[
\mathcal{M}(\sss\underline{\nu},\underline{\lambda},\zeta)
:=
\bigl\{\III\twoheadrightarrow F\mid F\in \catAzeta\bigr\}.
\]
Then we have a natural isomorphism between $\mathfrak{M}(\sss\underline{\nu},\underline{\lambda},\zeta)$ and $\mathcal{M}(\sss\underline{\nu},\underline{\lambda},\zeta)$. 
\end{NB}
The torus action on $Y_\sigma$ induces a torus action on the moduli $\mathfrak{M}^{\mathrm{ncDT}}(\zeta,\underline{\nu},\underline{\lambda})$.
We will compute the Euler characteristic by counting the number of torus fixed points. 
For a generic $\zeta$, the core $\catAinfzeta$ of the t-structure is isomorphic to the category of $A_\sigma^\zeta$-modules, 
where $A_\sigma^\zeta$ is associated with a quiver with a potential. 
Hence, we can describe the torus fixed point set on $\mathfrak{M}^{\mathrm{ncDT}}(\zeta,\underline{\nu},\underline{\lambda})$ in terms of a {\it crystal melting model} (\cite{ORV,ooguri-yamazaki}), which we have studied in \cite{open_3tcy}.
In fact, a particle in the grand state crystal gives a weight vector in 
$\PP$ with respect to the torus action, and a crystal obtained by removing a finite number of particles from the grand state crystal gives a torus fixed point in $\mathfrak{M}^{\mathrm{ncDT}}(\zeta,\underline{\nu},\underline{\lambda})$ (\S \ref{subsec_crystal_and_fixedpoint}). 
The invariants in this paper agree with the ones defined in \cite{open_3tcy}.

Finally, we provide explicit formulas for the generating functions of the Euler characteristic version of the open commutative and non-commutative Donaldson-Thomas invariants 
using vertex operator method, following \cite{ORV}, \cite{young-mckay} and \cite{bryan-young}\footnote{During preparing this paper, the author was informed that Piotr Sulkowski and Benjamin Young provide similar computations independently (\cite{Sulkowski,Young}). }. 
The order of the vertex operators is determined by the chamber in which the parameter $\zeta$ is. 
Hence we can understand the wall-crossing formula as the commutator relation of the vertex operators.

In Szendroi's original non-commutative Donaldson-Thomas theory (\cite{szendroi-ncdt}) the moduli spaces admit symmetric obstruction theory and the invariants are defined as the {\it virtual counting} of the moduli spaces in the sense of Behrend-Fantechi (\cite{behrend-fantechi-intrinsic})\footnote{Virtual counting coincides with the weighted Euler characteristic weighted by the Behrend function.}.
In the case when $\nu_+=\nu_-=\emptyset$, we show that the moduli space $\mathfrak{M}^{\mathrm{ncDT}}(\zeta,\underline{\emptyset},\underline{\lambda})$ admits a symmetric obstruction theory (\S \ref{subsec_potential}). 
Using the result in \cite{behrend-fantechi}, we can verify that the virtual counting coincide with the (non-weighted) Euler characteristics up to signs as in \cite{szendroi-ncdt,ncdt-brane,nagao-nakajima,3tcy} (\S \ref{subsec_w}).
We can also compute the generating function of the weighted (or non-weighted) Euler characteristics using Joyce-Song's theory (\cite{joyce-song}, or \cite{joyce-4}).

\smallskip

The plan of this paper is as follows: 
Section \ref{sec_1} contains basic observations on the core $\catAinfzeta$ of the t-structure of the derived category.
In Section \ref{sec_2}, the definition of Euler characteristic version of open non-commutative Donaldson-Thomas invariants is provided.
Then, we compute the generating function using vertex operators in Section \ref{sec_3}. 
Finally, we study open Donaldson-Thomas invariants and topological vertex as  ``limits'' of open non-commutative Donaldson-Thomas invariants in Section \ref{sec_5}.
Section \ref{sec_2}, Section \ref{sec_3} and Section \ref{sec_5} are the main parts of this paper.
In Section \ref{sec_4} we construct the moduli spaces used in Section \ref{sec_2} and \ref{sec_5} to define invariants. 
Moreover, we construct symmetric obstruction theory on the moduli space in the case of $\nu_+=\nu_-=\emptyset$ in Section \ref{subsec_potential}. The relation between weighted Euler characteristic and Euler characteristic is discussed in Section \ref{subsec_w}.
Throughout this paper we work on the half of the whole space of stability parameters. We will have a discussion about the other half of the stability space in Section \ref{subsec_NCPT}.
The computation in Section \ref{sec_3} depends on an explicit combinatorial description of the derived equivalence. We leave it until Section \ref{sec_7} since it is very technical.

\subsection*{Acknowledgement}
The author is supported by JSPS Fellowships for Young Scientists (No.\ 19-2672).
He thanks to Jim Bryan for letting him know the result of \cite{bryan-young} and recommending him to apply the vertex operator method  in the setting of this paper.
He also thanks to Osamu Iyama, Hiroaki Kanno, Hiraku Nakajima, Piotr Sulkowski, Yukinobu Toda, Masahito Yamazaki and Benjamin Young for useful comments.

This paper was written while the author is visiting the University of Oxford.
He is grateful to Dominic Joyce for the invitation and to the Mathematical Institute for hospitality.

\section*{Notations}
Let $\Zh$ denote the set of half integers and $L$ be a positive integer.
We set $I:=\Z/L\Z$ and $I_{\mathrm{h}}:=\Zh/L\Z$. The two natural projections $\Z\to I$ and $\Zh\to I_{\mathrm{h}}$ are denoted by the same symbol $\pi$.
We sometimes identify $I$ and $\Ih$ with $\{0,\ldots,L-1\}$ and $\{1/2,\ldots,L-1/2\}$ respectively.
The symbols $n$, $h$, $i$ and $j$ are used for elements in $\Z$, $\Zh$, $I$ and $\Ih$ respectively.

\smallskip

Throughout this paper, the following data play crucial roles:
\begin{itemize}
\item a map $\sigma\colon \Zh\to \{\pm\}$, which determines the crepant resolution $Y_\sigma\to X$ and the non-commutative crepant resolution $A_\sigma$, 
\item a pair of Young diagram $\unu=(\nu_+,\nu_-)$ and an $L$-tuple of Young diagrams $\ulam=(\lambda^{(1/2)},\ldots,\lambda^{(L-1/2)})$, which determines the ``asymptotic behaviors" of (complexes of) sheaves we will count, 
\item a stability parameter $\zeta$, which determines the t-structure where we will work on, and
\item a bijection $\theta\colon\Zh\to\Zh$, which determines the chamber where the stability parameter $\zeta$ is.
\end{itemize}
We sometimes identify a Young diagram $\mu$ with a map $\mu\colon\Zh\to \{\pm\}$ such that $\mu(h)=-$ for $h\ll 0$ and $\mu(h)=+$ for $h\gg 0$\footnote{Such a map $\mu$ is called a {\it Maya diagram}. See, for example, \cite[\S 2]{nagao-frenkel-kac} for the correspondence between a Young diagram and a Maya diagram.}.
We identify an $L$-tuple of Young diagrams $\ulam$ with a map $\ulam\colon \Zh\to\{\pm\}$ by 
\[
\ulam(h)=\lambda^{(\pi(h))}\biggl(\frac{h-\pi(h)}{L}+\frac{1}{2}\biggr).
\]

We define the following categories:
\begin{description}
\item[$\cohy$]: the Abelian category of coherent sheaves on $Y_\sigma$, 
\item[$\cohcy$] : the full subcategory of $\cohy$ consisting of coherent sheaves with compact supports, 
\item[$\dcohy$] : the bounded derived category $\cohy$, 
\item[$\dcohcy$] : the full subcategory of $\dcohy$ consisting of complexes with compactly supported cohomologies, 
\item[$\moda$] : the Abelian category of finitely generated left $A_\sigma$-modules, 
\item[$\modfa$] : the full subcategory of $\moda$ consisting of finite dimensional modules, 
\item[$\dmoda$] : the bounded derived category of $\moda$, 
\item[$\dmodfa$] : the full subcategory of $\dmoda$ consisting of complexes with finite dimensional cohomologies.
\end{description}

\section{T-structure and chamber structure}\label{sec_1}
\subsection{Non-commutative and commutative crepant resolutions}\label{suesec_11}
Let $\sigma$ be a map from $\Ih$ to $\{\pm\}$. 
In \cite{3tcy}, following \cite{hanany-vegh}, we introduced a quiver with a potential $A_\sigma=(Q_\sigma,w_\sigma)$, which is a non-commutative crepant resolution of $X$ (\cite{3tcy}). 
First, we set 
\begin{align*}
H(\sigma)&:=\bigr\{n\in\Z\mid \sigma(n-1/2)=\sigma(n+1/2)\bigl\},\quad I_H(\sigma):=\pi(H(\sigma)),\\
S(\sigma)&:=\bigr\{n\in\Z\mid \sigma(n-1/2)\neq \sigma(n+1/2)\bigl\},\quad I_S(\sigma):=\pi(S(\sigma)).
\end{align*}
The symbol $H$ and $S$ represent ``hexagon" and ``square" respectively. We use such notations since an element in each set corresponds to a hexagon or square in the dimer model (see \cite[\S 1.2]{3tcy}).
The vertices of $Q_{\sigma}$ are parametrized by $I$ and the arrows are given by 
\[
\Biggr(\,\bigsqcup_{j\in\Ih}h_j^+\Biggr)
\,\sqcup\,
\Biggr(\,\bigsqcup_{j\in\Ih}h_j^-\Biggr)
\,\sqcup\,
\Biggr(\bigsqcup_{i\in I_H(\sigma)}r_i\Biggr).
\]
Here $h^+_j$ (resp. $h^-_j$) is an edge from $j-1/2$ to $j+1/2$ (resp. from $j+1/2$ to $j-1/2$) and $r_i$ is an edge from $i$ to itself.
See \cite[\S 1.2]{3tcy} for the definition of the potential $w_{\sigma}$.
\begin{NB}
\begin{ex}
In Figure \ref{fig27}, we show the quiver $Q_{\sigma,\mathrm{id}}$ for $\sigma$  as in Example \ref{ex2}.
\begin{figure}[htbp]
  \centering
  \input{fig27.tpc}
  \caption{the quiver $w_{\sigma,\theta}$}
\label{fig27}
\end{figure}
\end{ex}
\begin{rem}
\begin{itemize}
\item The center of $A$ is isomorphic to $R:=\C[x,y,z,w]/(xy=z^{L_+}w^{L_-})$.
In \cite[Theorem 1.14 and 1.19]{3tcy}, we showed that $A$ is a non-commutative crepant resolution of $X=\mathrm{Spec}\, R$.
\item The affine $3$-fold $X$ is toric. 
In fact, 
\[
T=\mathrm{Spec}\,\tilde{R}:=\mathrm{Spec}\,\C[x^\pm,y^\pm,z^\pm,w^\pm]/(xy=z^{L_+}w^{L_-})
\]
is a $3$-dimensional torus.
The torus $T$ acts on $A$ by $\tilde{R}\to $
\end{itemize}
\end{rem}
\end{NB}

Let $P_i$ (resp. $S_i$) be the projective (resp. simple) $A_\sigma$-module corresponding to the vertex $i$.
Let $K_{\mathrm{num}}(\modfa)$ be the numerical Grothendieck group of $\modfa$, which we identify with $\Z^I$ by the natural basis $\{[S_i]\}$.
We put $\delta=(1,\ldots,1)\in K_{\mathrm{num}}(\modfa)$.

We identity the dual space $(K_{\mathrm{num}}(\modfa)\otimes \R)^*$ with $\R^I$ by the dual basis of $\{[S_i]\}$. 
Take $\zeta^\circ_{\mathrm{cyc}}:=(-L+1,1,1,\ldots,1)\in\R^I$.
\begin{thm}[\protect{\cite{ishii-ueda}}]
The moduli space of $\zeta^\circ_{\mathrm{cyc}}$-stable \textup{(}$=$ $\zeta_{\mathrm{cyc}}^\circ$-semistable\textup{)} $A_\sigma$-modules with dimension vectors $=\delta$ gives a crepant resolution of $X$. 
\end{thm}
Let $Y_{\sigma}$ denote this crepant resolution. 
\begin{thm}[\protect{\cite[\S 1]{3tcy}\label{thm_perv}\footnote{It is known by \cite{Mozgovoy1,Bocklandt} that a quiver with a potential given from a brane tiling satisfying the ``consistency condition" (\cite{ncdt-brane,Davison,Broomhead,Ishii_Ueda2}) is a non-commutative crepant resolution over its center (\cite{vandenbergh-nccr}). The claim of this theorem is a little bit stronger, i.e. $A_\sigma$ is given by the construction in \cite{vandenbergh-3d} and hence we have $\perv\simeq \moda$. We will use this equivalence of the Abelian categories in Section \ref{sec_5}.}, see \S \ref{subsec_openDT}}]
We have a derived equivalence between $\dcohy$ and $\dmoda$, which restricts to an equivalence between $\dcohcy$ \textup{(}resp. $\perv$\textup{)} and $\dmodfa$ \textup{(}resp. $\moda$\textup{)}.
\end{thm}

\subsection{Stability condition and tilting}\label{subsec_stability_and_tilting}
For $\zeta^\circ\in (K_{\mathrm{num}}(\modfa)\otimes \R)^*\simeq \R^I$ such that $\zeta^\circ\cdot \delta=0$, we define the group homomorphism
\[
Z_{\zeta^\circ}\colon K_{\mathrm{num}}(\mathrm{mod}_{\mathrm{fin}}A_\sigma)\to\C
\]
by 
\[
Z_{\zeta^\circ}(\mathbf{v}):=\bigl(-\zeta^\circ+\eta\sqrt{-1}\bigr)\cdot \mathbf{v}
\]
where $\eta=(1,\ldots,1)\in (K_{\mathrm{num}}(\modfa)\otimes \R)^*\simeq \R^I$. 
Then $(\modfa, Z_{\zeta^\circ})$ gives a stability condition on $\dmodfa$ in the sense of Bridgeland (\cite{bridgeland-stability}).

For a pair of real numbers $t_1>t_2$, let 
$\D{t_1}{t_2}$ be the full subcategory of $\dmodfa$ consisting of elements whose Harder-Narasimhan factors have phases less or equal to $t_1\pi$ and larger than $t_2\pi$. 
The following claims are standard  (see \cite{bridgeland-stability}):
\begin{lem}
\textup{(1)} $\D{t_1+1}{t_2+1}=\D{t_1}{t_2}[1]$ where $[1]$ represents the shift in the derived category.

\noindent\textup{(2)} $\D{t}{t-1}$ is a core of a t-structure for any $t$.

\noindent\textup{(3)} $\D{1}{0}=\modfa$.

\noindent\textup{(4)} For $t>s>t-1$, the pair of subcategories 
\[
\Bigl(\D{t}{s}, \D{s}{t-1}\Bigr)
\]
gives a {\it torsion pair} (\cite[Definition 2.4]{bridgeland_t_str}) for the Abelian category $\D{t}{t-1}$.
 
\noindent\textup{(5)} For $t>s>t-1$, $\D{s}{s-1}$ is obtained from $\D{t}{t-1}$ by tilting with respect to the torsion pair above (\cite{HRS_tilting}, \cite[Proposition 2.5]{bridgeland_t_str}), i.e. 
\begin{align*}
&\D{s}{s-1}=\\
&\Bigl\{
E\in\dmodfa\,\Big|\, H^0_{\D{t}{t-1}}(E)\in \D{s}{t-1},\ 
H^{1}_{\D{t}{t-1}}(E)\in \D{t}{s}\Bigr\},\\
&\D{t}{t-1}=\\
&\Bigl\{
E\in\dmodfa\,\Big|\, H^{0}_{\D{s}{s-1}}(E)\in \D{s}{t-1},\ 
H^{-1}_{\D{s}{s-1}}(E)\in \D{t-1}{s-1}
\Bigr\},
\end{align*}
where $H^*_{\D{t}{t-1}}(-)$ represents the cohomology with respect to the t-structure corresponding to $\D{t}{t-1}$. 
\end{lem}
\begin{lem}\label{lem_noether}
The algebra $A_\sigma$ is \textup{(}left-\textup{)}Noetherian.
\end{lem}
\begin{proof}
In \cite{3tcy}, it is shown that $A_\sigma$ is isomorphic to $f_*\mathrm{End}V$ for a vector bundle $V$ on $Y_\sigma$, where $f$ is the contraction $Y_\sigma\to X$. Since $f$ is proper, $A_\sigma$ is finitely generated as an $\mathcal{O}_X$-module. Hence $A_\sigma$ is Noetherian. 
\end{proof}
\begin{prop}\label{prop_torsion_pair}
For $0<t<1$ we put
\[
\D{1}{t}^\bot:=\{E\in \moda\mid \mathrm{Hom}_{A_\sigma}(F,E)=0,\forall F\in \D{1}{t}\}.
\]
Then the pair of full subcategory $(\D{1}{t},\D{1}{t}^\bot)$ gives a torsion pair in $\moda$.
\end{prop}
\begin{proof}
We will prove that every object $F\in \moda$ fits into a short exact sequence
\[
0\to E\to F\to G\to 0
\]
for some pair of objects $E\in \D{1}{t}$ and $G\in \D{1}{t}^\bot$. 

By Lemma \ref{lem_noether}, $F$ has the maximal finite dimensional submodule $F^{(1)}$. Let $F^{(2)}$ denote the cokernel of the inclusion $F^{(1)}\hookrightarrow F$. 
Note that $\mathrm{Hom}_{A_\sigma}(X,F^{(2)})=0$ for any finite dimensional $A_\sigma$-module $X$.

Let 
\[
0\to F^{(3)}\to F^{(1)}\to F^{(4)}\to 0
\]
be the exact sequence such that $F^{(3)}\in \D{1}{t}$ and $F^{(4)}\in \D{t}{0}$.Note that for any $X\in \D{1}{t}$ we have $\mathrm{Hom}_{A_\sigma}(X,F^{(4)})=0$. 

Let $F^{(5)}$ denote the cokernel of the inclusion $F^{(3)}\hookrightarrow F$.
Then we have the following exact sequence:
\[
0\to F^{(4)}\to F^{(5)}\to F^{(2)}\to 0.
\]
This implies $\mathrm{Hom}_{A_\sigma}(X,F^{(5)})=0$ for any $X\in \D{1}{t}$. 
Put $E:=F^{(3)}$ and $G:=F^{(5)}$, then the claim follows.
\end{proof}
\begin{NB}
\begin{lem}
The natural inclusion $\modfa\subset \moda$ has a right adjoint. 
\end{lem}
\begin{proof}
Since $A_\sigma$ is Noetherian, an element $E\in \moda$ has the maximal finite dimensional submodule $E^{\mathrm{fin}}$. 
The functor $E\mapsto E^{\mathrm{fin}}$ gives a right adjoint of the inclusion.
\end{proof}
A right admissible subcategory of $A$ is a full subcategory $B\subset A$ such that the inclusion functor has a right adjoint (\cite[Definition 2.1]{bridgeland-flop}).
\begin{cor}
For $0<t<1$, $\D{1}{t}$ is a right admissible full subcategory of $\dmoda$.
\end{cor}
\end{NB}
\begin{defn}
For $0<t<1$ let $\Dinf{t}{t-1}$ denote the core of the t-structure given from $\moda$ by tilting with respect to the torsion pair in Proposition \ref{prop_torsion_pair}, i.e.
\[
\Dinf{t}{t-1}=
\Bigl\{
E\in\dmoda\,\Big|\, H^{0}_{\moda}(E)\in \Dinf{t}{0},\ 
H^{1}_{\moda}(E)\in \Dinf{1}{t}
\Bigr\}.
\]
\end{defn}
\begin{NB}
The following lemma is also easy to verify:
\begin{lem}
Let $E_1\to E_2\to E_3\to E_1[1]$ be a distinguished triangle in $\dmodfa$ and $E_2\in \D{s}{t-1}$ for $t>s>t-1$. Then, the triangle gives an exact sequence in $\D{t}{t-1}$ if and only if it gives an exact sequence in $\D{s}{s-1}$.
In particular, in such a case we have $E_1,E_3\in \D{s}{t-1}$.
\end{lem}
I'm not sure that the following is true.
Moreover, have a torsion pair
\[
\Bigl(\Dinf{0}{t}, \Dinf{t}{1}\Bigr)
\]
for the Abelian category $\moda$ such that
\[
\Dinf{0}{t}\cap \modfa=\D{0}{t},\quad \Dinf{t}{1}\cap \modfa=\D{t}{1}.
\]
Let $\Dinf{t}{t+1}$ be the Abelian category given by tilting.
\end{NB}

We have the following bijection:
\begin{equation*}\label{eq_zeta_bij}
\begin{array}{ccc}
\{
(\zeta^\circ,T)\mid \zeta^\circ\cdot \delta=0,\ T\in\R
\}
&\overset{\sim}{\longrightarrow}
&\R^I\simeq (K_{\mathrm{num}}(\modfa)\otimes \R)^*\\
(\zeta^\circ,T)
&\longmapsto
&\zeta^\circ-T\eta.
\end{array}
\end{equation*}
The inverse map is given by
\[
T:=-\zeta\cdot\eta/L,\quad
\zeta^\circ:=\zeta+T\eta 
\]
for $\zeta\in(K_{\mathrm{num}}(\modfa)\otimes \R)^*$.
For a fixed $T$, take $0<t<1$ such that $\tan(t\pi)=1/T$. 
Note that for an element $V\in\modfa$ we have
\begin{equation}\label{eq_stability}
\phi_{{Z}_{\zc}}(V)<t\pi\iff \zeta([V])<0
\end{equation}
where $\phi_{{Z}_{\zc}}(V):=\mathrm{arg}({Z}_{\zc}([V]))$.
\begin{defn}\label{defn_catA}
\[
\catAzeta:=\D{t}{t-1},\quad \catAinfzeta:=\Dinf{t}{t-1}.
\]
\end{defn}
\begin{NB}
\begin{figure}[htbp]
  \centering
  \input{chambers.tpc}
  \caption{$\catAzeta$}
  \label{fig_chambers}
\end{figure}
\end{NB}
\begin{rem}
We have the natural action $\mathrm{rot}$ of $\R$ on the space of Bridgeland's stability conditions given by rotation of the complex line which is the target of the central charge. 
We can embed $(K_{\mathrm{num}}(\modfa)\otimes \R)^*$ into the space of Bridgeland's stability conditions by 
\[
\zeta\mapsto Z_\zeta:=\mathrm{rot}_t(Z_{\zeta^\circ}).
\]
Note that $\D{t}{t-1}$ agrees with $\mathcal{D}^{\zeta}_{\mathrm{fin}}[1,0)$. 
This is the reason why we call $\zeta$ a stability parameter, although we will use the former description since it is more convenient in our argument.
\end{rem}

\subsection{Chamber structure}\label{subsec_13}
A stability parameter $\zeta$ is said to be generic if there is no ${Z}_{\zeta^\circ}$-semistable objects with phase $t$. 
Then we get a chamber structure in $(K_{\mathrm{num}}(\modfa)\otimes \R)^*\simeq\R^I$. 
\begin{prop}
The chamber structure coincides with the affine root chamber structure of type $A_{L-1}$.
\end{prop}
\begin{proof}
A ${Z}_{\zeta^\circ}$-semistable object $V$ has the phase $t$ if and only if $\zeta([V])=0$ and so the genericity in this paper agrees with the one in \cite{3tcy}.
Then the claim follows from \cite[Proposition 2.10, Corollary 2.12]{3tcy}.
\end{proof}
Here we give a brief review for the affine root system of type $A_{L-1}$. 
We call $P:=K_{\mathrm{num}}(\modfa)$ the {\it root lattice} and $\alpha_i:=[S_i]\in P$ a {\it simple root}. 
For $h,h'\in\Zh$, we define $\alpha_{[h,h']}\in P$ by 
\[
\alpha_{[h,h']}:=\begin{cases}
0 & h=h', \\
\alpha_{\pi(h+1/2)}+\cdots+\alpha_{\pi(h'-1/2)} & h<h',\\
-\alpha_{\pi(h-1/2)}-\cdots-\alpha_{\pi(h'+1/2)} & h>h'
\end{cases}
\]
and $\delta:=\alpha_0+\cdots\alpha_{L-1}$. 
We set
\[
\Lambda :=\{\alpha_{[h,h']}\mid h\neq h'\},\quad
\Lambda^+  :=\{\alpha_{[h,h']}\mid h< h'\},\quad
\Lambda^-  :=\{\alpha_{[h,h']}\mid h> h'\}
\]
and
\[
\Lambda^{\mathrm{re}}  :=\{\alpha_{[h,h']}\mid h\not\equiv h'\,(\mathrm{mod}\,L)\},\quad
\Lambda^{\mathrm{im}}  :=\{m\delta\mid m\neq 0\}.
\]
An element in $\Lambda$ (resp. $\Lambda^{+}$, $\Lambda^{-}$, $\Lambda^{\mathrm{re}}$, $\Lambda^{\mathrm{im}}$) is called a {\it root} (resp. {\it positive} root, {\it negative} root, {\it real} root, {\it imaginary} root). 
Note that $\Lambda=\Lambda^+\sqcup \Lambda^-=\Lambda^{\mathrm{re}}\sqcup\Lambda^{\mathrm{im}}$. We put $\Lambda^{\mathrm{re},+}:=\Lambda^{\mathrm{re}}\cap\Lambda^+$ and define $\Lambda^{\mathrm{re},-}$, $\Lambda^{\mathrm{im},+}$ and $\Lambda^{\mathrm{im},-}$ in the same way. 

For a root $\alpha$, let $W_\alpha$ denote the hyperplane in $(K_{\mathrm{num}}(\modfa)\otimes \R)^*$ given by 
\[
W_\alpha:=\{\zeta\in(K_{\mathrm{num}}(\modfa)\otimes \R)^*\mid \zeta\cdot \alpha=0\}.
\]
The walls in the affine root chamber structure of type $A_{L-1}$ is given by
\[
W_\delta\cup\bigcup_{\alpha\in \Lambda^{\mathrm{re},+}}W_\alpha.
\]
Throughout this paper, we work on the area below the wall $W_\delta$, i.e. on the area $\{\zeta\mid \zeta\cdot \delta<0\}$.

\subsection{Parametrization of chambers}\label{subsec_14}
Let $\Theta$ denote the set of bijections $\theta\colon \Zh\to\Zh$ such that 
\begin{itemize}
\item $\theta(h+L)=\theta(h)+L$ for any $h\in\Zh$, and
\item $\theta(1/2)+\cdots+\theta(L-1/2)=1/2+\cdots +(L-1/2)=L^2/2$.
\end{itemize}
We have a natural bijection between $\Theta$ and the set of chambers in the area $\{\zeta\mid \zeta\cdot \delta<0\}$.
An element $\zeta_\theta$ in the chamber $C_\theta$ corresponding to $\theta\in\Theta$ satisfies the following condition:
\[
\alpha_{[h,h']}\cdot \zeta_\theta<0
\iff
\theta(h)<\theta(h')
\]
for any $h<h'$.
For $\theta\in \Theta$ and $i\in I$, we define $\alpha(\theta,i)\in P$ by 
\[
\alpha(\theta,i):=\alpha_{[\theta(n-1/2),\theta(n+1/2)]}\quad (\pi(n)=i).
\] 
Then the chamber $C_\theta$ is adjacent to the walls $W_{\alpha(\theta,i)}$ and we have $\zeta_\theta\cdot \alpha(\theta,i)<0$ for $\zeta_\theta\in C_\theta$.

Let $\theta_i\colon \Zh\to\Zh$ be the bijection given by
\[
\theta_i(h)=
\begin{cases}
h+1 & \pi(h+1/2)=i,\\
h-1 & \pi(h-1/2)=i,\\
h & \text{otherwise}.
\end{cases}
\]
Then we have $\alpha(\theta,i)=-\alpha(\theta\circ\theta_i,i)$ and the chambers $C_{\theta}$ and $C_{\theta\circ\theta_i}$ are separated by the wall $W_{\alpha(\theta,i)}=W_{\alpha(\theta\circ\theta_i,i)}$.

\subsection{Mutation}\label{subsec_15}
Assume that $T<0$ and
$\zeta^\circ$ is such that $(\zeta^\circ,T')$ is not on an intersection of two walls for any $T'\in\R$.
Let $\{T_r\}$ ($T_1<T_2<\cdots<0$) be the set of all the parameters $T_r<0$ such that $(\zeta^\circ,T_r)$ is not generic.
According to the argument at the end of the previous subsection, we have the  sequence $\{i_r\}$ of elements in $I$ such that  $(\zeta^\circ,T_r)$ for any $r$ is on the wall $W_{\alpha_r}$ for 
\[
\alpha_r:=\alpha(\theta_{i_{1}}\circ\cdots\circ\theta_{i_{r-1}},i_r).
\]
Take the minimal positive integer $R$ such that $T<T_R$ and put $A_\sigma^\zeta:=A_{\sigma\circ\theta_{i_1}\circ\cdots\circ\theta_{i_{R-1}}}$. 
Using this notation we have the following equivalencies of the Abelian categories:
\begin{prop}\label{prop_t_str}
\[
\catAzeta\simeq\mathrm{mod}_{\mathrm{fin}}A_\sigma^\zeta,\quad
\catAinfzeta\simeq\mathrm{mod}A_\sigma^\zeta.
\]
\end{prop}
\begin{proof}
We have the derived equivalence between $A_{\sigma\circ\theta_{i_1}\circ\cdots\circ\theta_{i_{r-1}}}$ and $A_{\sigma\circ\theta_{i_1}\circ\cdots\circ\theta_{i_{r}}}$ obtained by the tilting generator as in \cite[Proposition 3.1]{3tcy}. 
It is easy to see that, under this equivalence, the module category of $A_{\sigma\circ\theta_{i_1}\circ\cdots\circ\theta_{i_{r}}}$ is obtained from the one of $A_{\sigma\circ\theta_{i_1}\circ\cdots\circ\theta_{i_{r-1}}}$ by tilting with respect to the torsion pair obtained by the simple module.

Combine with the descriptions in \S \ref{subsec_stability_and_tilting}, we can see the claim by induction with respect to $r$. 
\end{proof}
Let $S_i^\zeta$ be the simple $A_\sigma^\zeta$-module associated to the vertex $i$.
For $\zeta\in C_\theta$, we have
\begin{equation}\label{eq_alpha}
[S_i^\zeta]=\alpha(\theta,i)\in K_{\mathrm{num}}(\modfa)
\end{equation}
under the induced isomorphism 
\[
K_{\mathrm{num}}(\modfa)\simeq K_{\mathrm{num}}(\mathrm{mod}_{\mathrm{fin}}A_\sigma^\zeta),
\]
\begin{NB}
We put 
\[
\catAinfzeta:=\mathrm{mod}A^{\zeta}_\sigma.
\]
This is obtained from $\moda$ by tilting with respect to the right (left?) admissible subcategory
\[
\D{1}{t}\subset \modfa\subset\moda.
\]
\end{NB}

\begin{NB}
We say $\zeta^\circ$ is generic if $(\zeta^\circ,t)$ does not on an intersection of two chambers for any $t$.
Let $\{t_r\} (t_1<t_2<\cdots<0)$ be the set of all the parameter $t_r$ such that $(\zeta^\circ,t_r)$ is not generic. 
Let $\alpha^r$ be the real root such that $(\zeta^\circ,t_r)$ is on the wall $W_{\alpha^r}$.
\begin{prop}
There exist a sequence $\{i_r\}$ of elements in $I$ such that
\[
\alpha^r=\phi_{i_{1}}\cdots\phi_{i_{r-1}}\alpha_{i_r}.
\]
\end{prop}
\begin{proof}
We can take $i_r$ inductively.
Assume that $(\zeta^\circ,t_r-\varepsilon)$ ($0<\varepsilon<<1$) is in a chamber surrounded by walls corresponding to 
\[
\phi_{i_{1}}\cdots\phi_{i_{r-1}}\alpha_{i} \quad (i\in I)
\]
and $(\zeta^\circ,t_r)$ is on the wall corresponding to
\[
\phi_{i_{1}}\cdots\phi_{i_{r-1}}\alpha_{i_r}.
\]
Then we can verify that $(\zeta^\circ,t_r+\varepsilon)$ ($0<\varepsilon<<1$) is in a chamber surrounded by walls corresponding to 
\[
\phi_{i_{1}}\cdots\phi_{i_{r}}\alpha_{i} \quad (i\in I).
\]
\end{proof}
\end{NB}

\section{Definition of the invariants}\label{sec_2}
\subsection{Ideal sheaf associated to Young diagrams}\label{subsec_21}
In this paper, we regard a Young diagram as a subset of $(\Z_{> 0})^2$.
For a Young diagram $\lambda$, let $\Lambda^x(\lambda)$ (resp. $\Lambda^y(\lambda)$ or $\Lambda^z(\lambda)$) be the subset of $(\Z_{> 0})^3$ consisting of the elements $(x,y,z)\in(\Z_{> 0})^3$ such that 
$(y,z)\in\lambda$ (resp. $(z,x)\in\lambda$ or $(x,y)\in\lambda$).
Given a triple $(\lambda_x, \lambda_y, \lambda_z)$ of Young diagrams, let $\Lambda^{\mathrm{min}}=\Lambda^{\mathrm{min}}_{\lambda_x, \lambda_y, \lambda_z}$ be the following subset of $(\Z_{> 0})^3$:
\[
\Lambda^{\mathrm{min}}:=\Lambda^x(\lambda_x)\cup\Lambda^y(\lambda_y)\cup\Lambda^z(\lambda_z)\subset(\Z_{> 0})^3.
\]
A subset $\Lambda$ of $(\Z_{> 0})^3$ is said to be a {\it $3$-dimensional Young diagram of type $(\lambda_x, \lambda_y, \lambda_z)$} if the following conditions are satisfied:
\begin{itemize}
\item if $(x,y,z)\notin\Lambda$, then $(x+1,y,z), (x,y+1,z), (x,y,z+1)\notin\Lambda$,
\item $\Lambda\supset\Lambda^{\mathrm{min}}$, and
\item $|\Lambda\backslash\Lambda^{\mathrm{min}}|<\infty$.
\end{itemize}
For a $3$-dimensional Young diagram $\Lambda\subset (\Z_{> 0})^3$, we define an ideal $I_\Lambda\subset\C[X,Y,Z]$ by
\[
I_\Lambda:=\bigoplus_{(x,y,z)\notin\Lambda}\C\cdot X^xY^yZ^z.
\]
This is invariant with respect to the torus action $T:=(\C^*)^3\curvearrowright \C^3$.
On the other hand, any torus invariant ideal can be described in this way.

Recall that the toric graph of $Y_\sigma$ has 
\begin{itemize}
\item $L$ vertices, 
\item $L-1$ closed edges, and 
\item $L+2$ open edges.
\end{itemize}
\begin{figure}[htbp]
  \centering
  \input{fig_toric_graph.tpc}
  \caption{A toric graph and Young diagrams}
  \label{fig_toric_graph}
\end{figure}
The set of torus invariant ideal sheaf on $Y$ is parametrized by the following data (\cite{mnop}):
\begin{itemize}
\item a pair of Young diagrams $\underline{\nu}=(\nu_+,\nu_-)$ and an $L$-tuple of Young diagrams $\underline{\lambda}=(\lambda^{(1/2)},\ldots,\lambda^{(L-1/2)})$ corresponding to $L+2$ open edges, 
\item an $(L-1)$-tuple of Young diagrams $\nu^{(1)},\ldots,\nu^{(L-1)}$ corresponding to closed edges, and 
\item an $L$-tuple of $3$-dimensional Young diagrams
$\Lambda^{(1/2)},\ldots,\Lambda^{(L-1/2)}$ corresponding to vertices such that $\Lambda^{(j)}$ is 
\begin{itemize}
\item of type $(\lambda^{(j)},\nu^{(j+1/2)},{}^{\mathrm{t}}\nu^{(j-1/2)})$ if $\sigma(j)=+$,
\item of type $(\lambda^{(j)},{}^{\mathrm{t}}\nu^{(j-1/2)},\nu^{(j+1/2)})$ if $\sigma(j)=-$,\label{item_weight}
\end{itemize}
where we put $\nu^{(0)}:=\nu_-$ and $\nu^{(L)}:=\nu_+$.
\end{itemize}
Let 
\[
\I(\unu,\ulam\sss;\sss\nu^{(1)},\ldots,\nu^{(L-1)})
\]
be the ideal associated with the $L$-tuple of $3$-dimensional Young diagrams
$\Lambda^{(1/2)},\ldots,\Lambda^{(L-1/2)}$ where
\begin{itemize}
\item $\Lambda^{(j)}=\Lambda_{\lambda^{(j)},\nu^{(j+1/2)},{}^{\mathrm{t}}\nu^{(j-1/2)}}^{\mathrm{min}}$ if $\sigma(j)=+$,
\item $\Lambda^{(j)}=\Lambda_{\lambda^{(j)},{}^{\mathrm{t}}\nu^{(j-1/2)},\nu^{(j+1/2)}}^{\mathrm{min}}$ if $\sigma(j)=-$.
\end{itemize}
We set
\[
\I_{\unu,\ulam}:=\I(\unu,\ulam\sss;\sss\emptyset,\ldots,\emptyset).
\]
\begin{NB}
Let
\[
(\sss\underline{\nu},\underline{\lambda}\sss)=\Bigl(\nu_+,\nu_-,\lambda_+^{(1)},\ldots,\lambda_+^{(L^+)},\lambda_-^{(1)},\ldots,\lambda_-^{(L^-)}\Bigr)
\]
an $(L+2)$-tuple of Young diagrams associated with the $L+2$ open edges in the toric graph.
We put empty Young diagrams on $L-1$ closed edges in the toric graph.
Note that for each vertex $j$ in the toric diagram we have torus invariant open subvariety $U_j\subset Y_\sigma$, which is isomorphic to $\C^3$. The three edges touching the vertex correspond to the three axes in $U_j\simeq\C^3$.
For the three Young diagrams on these three edges, we have the closed subscheme $Z_j:=Z_{\lambda_x, \lambda_y, \lambda_z}$ of $U_j\simeq \C^3$, which is also closed in $Y_\sigma$.
Note that the supports of $Z_j$ are disjoint with each other. 
Let $\I_{(\underline{\nu},\underline{\lambda})}$ denote the ideal sheaf on $Y_\sigma$ corresponding to the closed subscheme $\cup Z_j$. 
\end{NB}

\subsection{Open non-commutative Donaldson-Thomas invariants}
Assume $\zeta\in C_\theta$ for some $\theta\in\Theta$.
We put 
\begin{equation}\label{eq_III}
\III:=H^0_{\catAinfzeta}(\II).
\end{equation}
where $H^*_{\catAinfzeta}(-)$ represents the cohomology with respect to the t-structure corresponding to $\catAinfzeta$. 
\begin{ex}
\textup{(1)} In the case when
\[
\unu=\underline{\emptyset}:=(\emptyset,\emptyset),\quad 
\ulam=\underline{\emptyset}:=(\emptyset,\ldots,\emptyset),
\]
we have 
\[
P^\zeta_{\uemp,\uemp}=\I_{\uemp,\uemp}=\mathcal{O}_{Y_\sigma}
\]
for any $\zeta$.

\noindent\textup{(2)} In the case when $L_+=L_-=1$ and
\[
\unu=\underline{\emptyset}:=(\emptyset,\emptyset),\quad 
\ulam:=(\emptyset,\square),
\]
we have an exact sequence 
\[
0\to \I(\uemp,\ulam\sss;\sss\square)\to \I_{\uemp,\ulam}\to \mathcal{O}_C(-1)\to 0. 
\]
We can see that $\I(\uemp,\ulam\sss;\sss\square)$ does not have $\mathcal{O}_C(-n)$ as its quotient for $n>0$\footnote{Suppose that we have a surjection from $\I(\uemp,\ulam\sss;\sss\square)$ to $\mathcal{O}_C(-n)$. We may assume that the map is torus equivariant. Then the kernel is described by a pair of $3$-dimensional Young diagrams which is obtained by removing some boxes from the pair of $3$-dimensional Young diagrams associated to $\I(\uemp,\ulam\sss;\sss\square)$. Then we can see that we can not remove boxes so that the cokernel is $\mathcal{O}_C(-n)$ ($n>0$).}. 
This means that $\I(\uemp,\ulam\sss;\sss\square)$ is in $^{-1}\mathrm{Per}(Y/X)$. 
Hence 
\[
P^\zeta_{\unu,\ulam}=\I(\uemp,\ulam\sss;\sss\square)\neq \I_{\uemp,\ulam}.
\]
for $\zeta=(\zeta_{0},\zeta_{1})$ such that $\zeta_0,\zeta_1<0$.
\end{ex}
Take $\mathbf{v}\in K_{\mathrm{num}}(\catAzeta)\simeq \Z^I$.
\begin{defn}\label{defn_pair}
A $(\ztpair)$-pair of type $\mathbf{v}$ is a pair $(V,s)$ of an element $V\in \catAzeta$ with $[V]=\mathbf{v}$ and a surjection $s\colon\III\surj V$ in $\catAinfzeta$.

Two $(\ztpair)$-pairs $(V,s)$ and $(V',s')$ are said to be equivalent if there exists a isomorphism between $V$ and $V'$ compatible with $s$ and $s'$.
\end{defn}
\begin{thm}\label{thm_moduli}
There is a coarse moduli scheme $\MMM$ parameterizing equivalence classes of $(\ztpair)$-pairs $(V,s)$ with $[V]=\mathbf{v}$.
\end{thm}
The proof of this theorem is given in \S \ref{subsec_proof}.

\begin{NB}
We define $\PP\in\moda$ by
\[
\PP:=H^0_{\moda}(\II).
\]
\begin{defn}
An $\Apair$-pair is a pair of a finite dimensional $A_\sigma$-module $V$ and a homomorphism $s\colon \PP\to V$. 
\end{defn}
Let $\zeta\in(K_{\mathrm{num}}(\modfa)\otimes \R)^*$ be a stability parameter.
\begin{defn}
We say $(V,s)$ is $\zeta$-(semi)stable if 
\begin{enumerate}
\item[\textup{(A)}] for any nonzero submodule $0\neq V'\subseteq V$, we have
\[
\zeta\cdot\dimv V' \,(\le)\,0,
\]
\item[\textup{(B)}] for any proper submodule $V'\subsetneq V$ through which $s$ factors, we have
\[
\zeta\cdot\dimv V'\,(\le)\,\zeta \cdot\dimv V.
\]
\end{enumerate}
\end{defn}
\begin{rem}
The condition \textup{(B)} is equivalent to the following:
\begin{enumerate}
\item[\textup{(B')}] for any nonzero quotient module $V\twoheadrightarrow V''$ such that the composition $\PP\to V\twoheadrightarrow  V''$ is trivial, we have 
\[
0\,(\le)\,\zeta \cdot\dimv V''.
\]
\end{enumerate}
\end{rem}
\begin{thm}\label{thm_moduli}
There is a coarse moduli scheme $\MM$ parameterizing S-equivalence classes of $\zeta$-semistable pairs $(V,s)$.
\end{thm}
The proof of this theorem is given in \S \ref{subsec_proof}.

\begin{NB}
Let $A_\sigma$ be a non-commutative crepant resolution of $X$ (\cite{}). 
We identify the derived category of coherent sheaves on $Y$ and the one of $A_\sigma$-modules by the derived equivalence which maps $\mathrm{mod}A_\sigma$ to the category of perverse sheaves on $Y$. 
Let $H^*_{\mathrm{mod}A}(-)$ denote the cohomology with respect to the t-structure corresponding to $\mathrm{mod}A$. 
We set 
\[
P_{(\underline{\nu},\underline{\lambda})}:=H^0_{\mathrm{mod}A}(\mathcal{I}_{(\underline{\nu},\underline{\lambda})}).
\]
Let $\zeta=(\zeta_1,\ldots,\zeta_l)\in \R^l$ be a $l$-tuple of real numbers. 
\begin{defn}
We say $(V,s)$ is $\zeta$-(semi)stable if 
\begin{enumerate}
\item[\textup{(A)}] for any nonzero submodule $0\neq V'\subseteq V$, we have
\[
\zeta\cdot\dimv V' \,(\le)\,0,
\]
\item[\textup{(B)}] for any proper submodule $V'\subsetneq F$ which includes $\mathrm{Im}(s)$, we have
\[
\zeta\cdot\dimv V'\,(\le)\,\zeta \cdot\dimv V.
\]
\end{enumerate}
\end{defn}
\begin{thm}\label{thm_moduli}
There is a coarse moduli scheme parameterizing S-equivalence classed of $\zeta$-semistable pairs $(V,s)$.
\end{thm}
\end{NB}

\begin{defn}
We define the generating function of the Euler characteristic version of {open non-commutative Donaldson-Thomas invariants}
\[
\mathcal{Z}^{\mathrm{E-ncDT}}_{\ztpair}(q_0,\ldots,q_{L-1}):=\sum_\mathbf{v} e\Bigl(\MMM\Bigr)\cdot (\mathbf{q}_\theta)^{\mathbf{v}},
\]
where $(\mathbf{q}_\theta)^{\mathbf{v}}:=\prod(\mathbf{q}^{\alpha(\theta,i)})^{v_i}$\footnote{We use the monomial $\mathbf{q}^{\alpha(\theta,i)}$ since $\alpha(\theta,i)$ is the numerical class of the simple $A_\sigma^\zeta$-module $S_i^\zeta$ (see \eqref{eq_alpha}).}.
\end{defn}

\section{Torus fixed points, crystal melting and vertex operators}\label{sec_3}
\subsection{Crystal melting model and torus fixed points}\label{subsec_crystal_and_fixedpoint}
\begin{defn}
Let $\mu$ and $\mu'$ be two Young diagrams. 
We say $\mu\pg\mu'$ if the row lengths satisfy
\[
\mu_1\geq\mu'_1\geq\mu_2\geq\mu'_2\geq\cdots,
\]
and $\mu\mg\mu'$ if the column lengths satisfy
\[
\tenchi\mu_1\geq\tenchi\mu'_1\geq\tenchi\mu_2\geq\tenchi\mu'_2\geq\cdots.
\]
\end{defn}
\begin{defn}
Let $\Pi$ denote the set of all Young diagrams.
A {\it transition} $\V$ of Young diagrams of type $(\type)$ is a map $\V\colon \Z\to \Pi$ such that
\begin{itemize}
\item $\V(n)=\nu_-$ for $n\ll0$ and $\V(n)=\nu_+$ for $n\gg0$,
\item $\V(h-\ulam\circ\theta(h)/2)\overset{\sigma\circ\theta(h)}{\succ}\V(h+\ulam\circ\theta(h)/2)$.
\end{itemize}
\end{defn}
\begin{NB}
\begin{defn}
Let $\Pi$ denote the set of all Young diagrams.
For a transition $\V$ of Young diagram of type $(\type)$, we define the weight $w(\V)\in\Z^I$ of $\V$ by 
\[
w(\V)_i:=\sum_{\pi(n)=i}\sharp(\V(n)\backslash \V_{\mathrm{min}}(n)).
\]
\end{defn}
\begin{defn}
We put
\[
P(\sigma,\theta,\underline{\nu},\underline{\lambda})_i:=
\bigoplus_{
\begin{subarray}{c}
n\in \pi^{-1}(i)\\
(x,y)\notin \Vmin(n)
\end{subarray}}
\C\cdot v(n,x,y). 
\]
We define an $A_{\sigma\circ\theta}$-action on $P(\sigma,\theta,\underline{\nu},\underline{\lambda}):=\oplus P(\sigma,\theta,\underline{\nu},\underline{\lambda})_i$ by 
\begin{align*}
h_j^+(v(h-1/2,x,y))=v(h+1/2,x,y)
h_j^-
r_i
r_i
\end{align*}
\end{defn}
\end{NB}
\begin{defn}
For a transition $\V$ of Young diagram of type $(\type)$, we put 
\[
P(\V)_i:=\{(n,x,y)\in \Z\times (\Z_{\geq 0})^2\mid n\equiv i\ (\mr{mod} L),\ (x,y)\notin \V(n)\}
\]
and $P(\V):=\sqcup_i P(\V)_i$. 
We use the notation $p(n,x,y)$ for an element in $P(\V)$. 
\end{defn}
\begin{lem}[\protect{\cite[\S 3.3.3, Remark 3.7]{open_3tcy}}]
There is a transition $\Vmin=\Vmin^{\type}$ of Young diagrams of type $(\type)$ such that for any transition $\V$ of Young diagrams of type $(\type)$ we have $P(\Vmin)\supseteq P(\V)$. 
\end{lem}
\begin{proof}
In \cite{open_3tcy}, we use the notation $\nu$ and $\lambda$ instead of $\unu$ and $\ulam$.
In \cite[\S 3.3.3]{open_3tcy}, a map $G^\nu_{\sigma,\lambda,\theta}$ is given.
As is mentioned in \cite[Remark 3.7]{open_3tcy}, this map gives a sequence of Young diagrams $\Vmin^{\type}$, which satisfies the condition.
\end{proof}
\begin{defn}
A crystal of type $(\type)$ is a subset $P(\V)$ of $P(\Vmin)$ such that $|P(\Vmin)\backslash P(\V)|<\infty$.
\end{defn}
The lemma above claims that giving a transition $\V$ of Young diagram of type $(\type)$ is equivalent to giving a crystal $P(\V)$ of type $(\type)$.
\begin{defn}\label{defn_cryctal_rep}
Let $M(\V)=\oplus_iM(\V)_i$ be the vector space spanned by the elements in $P(\V)=\sqcup_iP(\V)_i$.
We define an $A_{\sigma\circ\theta}$-action on $M(\V)$ by 
\begin{align*}
h_j^+(p(h-1/2,x,y))&=
\begin{cases}
p(h+1/2,x,y) & (\sss\ulam\circ \theta(h)=+),\\
p(h+1/2,x+1,y) & (\sss\ulam\circ \theta(h)=-, \sigma\circ\theta(h)=-),\\
p(h+1/2,x,y+1) & (\sss\ulam\circ \theta(h)=-, \sigma\circ\theta(h)=+),\\
\end{cases}\\
h_j^-(p(h+1/2,x,y))&=
\begin{cases}
p(h-1/2,x,y) & (\sss\ulam\circ \theta(h)=-),\\
p(h-1/2,x+1,y) & (\sss\ulam\circ \theta(h)=+, \sigma\circ\theta(h)=-),\\
p(h-1/2,x,y+1) & (\sss\ulam\circ \theta(h)=+, \sigma\circ\theta(h)=+),\\
\end{cases}\\
r_i(p(n,x,y))&=
\begin{cases}
p(n,x+1,y) & (\sss\sigma\circ\theta(n-1/2)=\sigma\circ\theta(n+1/2)=+),\\
p(n,x,y+1) & (\sss\sigma\circ\theta(n-1/2)=\sigma\circ\theta(n+1/2)=-).
\end{cases}
\end{align*}
\end{defn}
\begin{prop}\label{prop_sec_6}
For $\zeta\in C_\theta$, we have $M(\Vmin^{\type})\simeq \III$ as an $A_{\sigma\circ\theta}$-module.
\end{prop}
The proof of this proposition is given in \S \ref{sec_7}.
\begin{rem}
The $A_{\sigma\circ\theta}$-module $M(\Vmin^{\type})$ coincides with $M^{\mathrm{max}}_{\sigma,\lambda,\nu,\theta}$ defined in \cite[\S 3.3.3]{open_3tcy}. 
\end{rem}

\begin{NB}
For a subset $C\subset P(\Vmin)$, let $M(C)\subset M(\Vmin)$ be a subspace spanned by the elements in $C$. 
\begin{lem}
Let $(\III \twoheadrightarrow V)\in \MM^T$ be a torus fixed point. 
Then the kernel of the map is described as $M(C)$ for a subset $C\in P(\Vmin)$.
\end{lem}
\begin{proof}
\end{proof}
\end{NB}
\begin{prop}
Let $(\III \twoheadrightarrow V)\in \MM$ be a torus fixed point. 
Then the kernel of the map is described as $M(\V)$ for a transition $\V$.
\end{prop}
\begin{proof}
Take a one parameter subgroup $\rho\colon T\to \prod\mathrm{GL}\Bigl(\Bigl(\III\Bigr)_{\hspace{-2pt}i}\,\Bigr)$ such that $\rho(t)*\III=t\cdot\III$.
Each element in $P(\Vmin)$ gives an eigenvector for $\rho$ and the eigenvalues are distinct from each other. 
Hence the kernel is spanned by a subset of $P(\Vmin)$.
We can verify that a subset of $P(\Vmin)$ gives an $A_\sigma^\zeta$-submodule of $M(\Vmin)$ if and only if it is a crystal.
\end{proof}
\begin{defn}
For a transition $\V$ of Young diagram of type $(\type)$,
we define the weight $\mathbf{v}(\V)\in\Z^I$ of $\V$ by 
\[
\mathbf{v}(\V)_i:=\sharp\{p(n,x,y)\in P(\Vmin)\backslash P(\V)\mid n\equiv i\ (\mr{mod}\,L)\}.
\]
\end{defn}
\begin{defn}
\[
\mathcal{Z}^{\mathrm{crystal}}_{\type}(q_0,\ldots,q_{L-1}):=\sum_{\text{$\V\colon$of type $(\type)$}} (\mathbf{q}_\theta)^{\mathbf{v}(\V)}.
\]
\end{defn}
\begin{cor}\label{cor_ncdt_crystal}
\[
\mathcal{Z}^{\mathrm{E-ncDT}}_{\ztpair}(q_0,\ldots,q_{L-1})=\mathcal{Z}^{\mathrm{crystal}}_{\type}(q_0,\ldots,q_{L-1}).
\]
\end{cor}

\subsection{Crystal melting and vertex operators}\label{subsec_crystal_and_vo}
Let $\K:=\C[q_i^{-1},q_i]]$ be the ring of Laurent formal power series with variables $q_i$ ($i\in I$) and $\Pi$ be the set of Young diagrams. 
We define the {\it Fock space} $\fock$ by
\[
\fock:=\bigoplus_{\lambda\in\Pi}\K\cdot \lambda. 
\]
Let $\langle\,-\,|\,-\,\rangle$ be the $\K$-bilinear inner product under which $\{\mu\}$ are orthonormal. 

We will use the "bra-ket" notation: 
\[
\langle\, \mu'\,|\,A\,|\,\mu\,\rangle:=\langle\, \mu'\,|\,A\mu\,\rangle=\langle\, {}^{\mathrm{t}}\hspace{-2pt}A\mu'\,|\,\mu\,\rangle,
\]
where $A$ is an endomorphism of $\fock$.
\begin{defn}
For $p\in\K$, we define the {\it vertex operators} 
$\Gamma^\pm_\pm(p)\colon\fock\to\fock$ by 
\[
\Gamma^\pm_+(p)\mu:=\sum_{\mu'\overset{\pm}{\succ}\mu}p^{|\mu|-|\mu'|}\mu',\quad 
\Gamma^\pm_-(p)\mu:=\sum_{\mu'\overset{\pm}{\prec}\mu}p^{|\mu|-|\mu'|}\mu'.
\]
\end{defn}
\begin{lem}\textup{(see \cite[Lemma 3.3]{young-mckay})}\label{lem_commutator}
For $p,p'\in\K$, we have
\[
[\Gamma^{\varepsilon_1}_{\iota_1}(p_1),\Gamma^{\varepsilon_2}_{\iota_2}(p_2)]=(1-\varepsilon_1\varepsilon_2p_1^{\iota_1}p_2^{\iota_2})^{-\iota_1\varepsilon_1\varepsilon_2\delta_{\iota_1+\iota_2}}.
\]
\end{lem}
For $w'\in\Z^I$, we set
\[
(\Z^I)_{\leq w'}:=\{w\in\Z^I\mid w_i\leq w'_i \text{ for all $i\in I$}\}.
\]
\begin{lem}\textup{(see \cite{young-mckay})}\label{lem_schur}
Let $J$ be a countable set and $w\colon J\to \Z^I$ be a map such that $w^{-1}((\Z^I)_{\leq w'})$ is finite for any $w'\in\Z^I$. 
Put $p_j:=\mathbf{q}^{w(j)}$. 
Note that for any map $\varepsilon\colon J\to\{\pm\}$ the operator $\prod_{j\in J}\Gamma^{\varepsilon(j)}_\pm(p_j)$ is well-defined.
Then we have
\begin{align*}
&\Bigl\langle\, \mu'\,\Big|\,\prod\Gamma^+_-(p_i)\,\Big|\,\mu\,\Bigr\rangle=\Bigl\langle\, \mu\,\Big|\,\prod\Gamma^+_+(p_i)\,\Big|\,\mu'\,\Bigr\rangle=s_{(\mu'\backslash\mu)}(p_i),\\
&\Bigl\langle\, \mu'\,\Big|\,\prod\Gamma^-_-(p_i)\,\Big|\,\mu\,\Bigr\rangle=\Bigl\langle\, \mu\,\Big|\,\prod\Gamma^-_+(p_i)\,\Big|\,\mu'\,\Bigr\rangle=s_{({}^{\mathrm{t}}\hspace{-1pt}\mu'\backslash{}^{\mathrm{t}}\hspace{-1pt}\mu)}(p_i)
\end{align*}
where $s_{(\mu'\backslash\mu)}$ and $s_{({}^{\mathrm{t}}\hspace{-1pt}\mu'\backslash{}^{\mathrm{t}}\hspace{-1pt}\mu)}$ denote the skew Schur functions.
\end{lem}
\begin{NB}
\begin{lem}
Let $\iota\colon \{1,\ldots,B\}\to \{\pm\}$ be a map and $J_b$ ($1\geq b\geq B$) be a countable set.
Take a map $w\colon \sqcup_{j\in J'}J_j\to \Z$ as above and put $p_i:=q^{w(i)}$. Then we have
\begin{align*}
&\Biggl(\,\prod_{j\in J_1}\Gamma^{\varepsilon(j)}_{\iota(1)}(q_j)\Biggr)
\cdots 
\Biggl(\,\prod_{j\in J_B}\Gamma^{\varepsilon(j)}_{\iota(B)}(q_j)\Biggr)\\
&=
\prod_{1\leq b<b'\leq B}\Biggl(\,\prod_{j\in J_b}\prod_{j'\in J_{b'}}
(1-\varepsilon(j)\varepsilon(j')pp')^{\varepsilon(j)\varepsilon(j')}
\Biggr)
\cdot\prod_{b\in \iota^{-1}(+)}\Biggl(\,\prod_{j\in J_b}\Gamma^{\varepsilon(j)}_{+}(q_j)\Biggr)
\cdot
\prod_{b\in \iota^{-1}(-)}\Biggl(\prod_{j\in J_b}\Gamma^{\varepsilon(j)}_{-}(q_j)\Biggr).
\end{align*}
\end{lem}
\end{NB}
\begin{NB}
We take a map $w=w(\theta,\ulam)\colon \Zh\to \Z^I$ such that 
\[
\mathbf{q}^{\ulam\circ\theta(n+1/2)\cdot w(n+1/2)-\ulam\circ\theta(n-1/2)\cdot w(n-1/2)}=q_{\pi(n)}
\]
and define 
\[
w(\Vmin):=\sum_{h}\Bigl|\Vmin(h-\ulam\circ\theta(h)/2)\backslash\Vmin(h+\ulam\circ\theta(h)/2)\Bigr|\cdot w(h).
\]
\end{NB}

Let $f(h)$ ($h\in\Zh$) be a sequence of operators.
If the composition of the operator
\[
\cdots \circ f(\theta^{-1}(h-1))\circ f(\theta^{-1}(h))\circ f(\theta^{-1}(h+1))\circ \cdots
\]
is well-defined, we denote this by 
\[
\prod_{h\in\Zh}^{\theta}f(h).
\]
We set $\alpha_h:=\alpha_{[1/2,h]}$ and $q_h:=\mathbf{q}^{\alpha_h}$ and define  the monomial 
\[
\mathbf{q}(\type):=\prod_{h\in\Zh}(q_h)^{|\Vmin(h-1/2)|-|\Vmin(h+1/2)|}.
\]
The following claim is clear from the definitions:
\begin{prop}\label{prop_crystal_vo}
\[
\mathcal{Z}^{\mathrm{crystal}}_{\type}(q_0,\ldots,q_{L-1})=
\Bigl\langle\, \nu_-\,\Big|\,\prod^{\theta}_{h\in\Zh}\Gamma^{\sigma(h)}_{\lambda(h)}(q_h)\,\Big|\,\nu_+\,\Bigr\rangle\cdot \mathbf{q}(\type).
\]
\end{prop}

\subsection{Computation of the generating function}\label{subsec_computation}
For a symmetric function $f=f(p_1,p_2,\ldots)$, let $f^*$ be the symmetric function given by $f^*(p_1,p_2,\ldots)=f(p_1^{-1},p_2^{-1},\ldots)$. 
For a subset $S\subset \Zh$, let $f(S\sss;\sss\mathbf{q})$ denote the symmetric function given by substituting $\{q_h\mid h\in S\}$ for $f$. 
We set 
\[
S_{\ulam}^{\,\iota}:=\{h\in\Zh\mid\ulam(h)=\iota\},\quad
S_{\sigma,\ulam}^{\,\varepsilon,\iota}:=\{h\in\Zh\mid\sigma(h)=\varepsilon,\, \ulam(h)=\iota\}.
\]
The following lemma is a direct consequence of Lemma \ref{lem_commutator}:
\begin{lem}\label{lem_3.16}
\begin{align*}
&\Bigl\langle\, \nu_-\,\Big|\,
\prod_{h\in S_{\ulam}^{+}}\Gamma^{+}_{\lambda(h)}(q_h)\cdot
\prod_{h\in S_{\ulam}^{-}}\Gamma^{-}_{\lambda(h)}(q_h)
\,\Big|\,\nu_+\,\Bigr\rangle=\\
&\sum_{{\nu_-\supseteq \nu_1\supseteq \nu_2\subseteq\nu_3\subseteq\nu_+}}
s_{\nu_-\backslash\nu_1}\bigl(S^{+,+}_{\sigma,\ulam}\sss;\sss\mathbf{q}\sss\bigr)\cdot 
\tenchi s_{\nu_1\backslash\nu_2}^*\bigl(S^{-,+}_{\sigma,\ulam}\sss;\sss\mathbf{q}\sss\bigr)\cdot 
s_{\nu_3\backslash\nu_2}\bigl(S^{+,-}_{\sigma,\ulam}\sss;\sss\mathbf{q}\sss\bigr)\cdot 
\tenchi s_{\nu_+\backslash\nu_3}^*(S^{-,-}_{\sigma,\ulam}\sss;\sss\mathbf{q}\sss\bigr)
\end{align*}
\end{lem}
For $\alpha=\alpha_{[h,h']}\in \Delta$, we set $\sigma(\alpha)=-\sigma(h)\sigma(h')$ and 
\[
\alpha_0:=\sharp\{m\in Z\mid h<mL<h'\}.
\]
We can easily verify the following:
\begin{lem}\label{lem_3.17}
For $\alpha\in\Delta^{\mathrm{re}}$ and for any $\ulam$, we have
\[
\sharp\{(h,h')\in(\Zh)^2\mid \alpha_{[h,h']}=\alpha,\ \ulam(h)=-,\ \ulam(h')=+\}=\alpha_0.
\]
\end{lem}
Let 
\[
M(1,t):=\prod_{m>0}(1-t^m)^{-m}
\]
be MacMahon function and 
\[
s_\lambda(t^{-\rho}):=s_\lambda(t^{1/2},t^{3/2},\ldots)
\]
be the specialization of Schur function.
The next equation follows from the hook length formula:
\begin{lem}
\[
\prod
(1-t^{m'-m})=
M(1,t)\cdot s_{\lambda^{(j)}}(t^{-\rho})
\]
where the product in the left hand side is taken over the set
\[
\{
(m,m')\mid m<m',\,\ulam(mL+j)=-,\,\ulam(m'L+j)=+
\}
\]
\end{lem}
\begin{NB}
For $\theta\in\Theta$, we put 
\[
q^\theta_i:=
\begin{cases}
q_{\theta(n-1/2)+1/2}\cdot q_{\theta(n-1/2)+3/2}\cdot\cdots\cdot q_{\theta(n+1/2)-1/2} &  (\theta(n-1/2)<\theta(n+1/2)),\\
q_{\theta(n-1/2)-1/2}^{-1}\cdot q_{\theta(n-1/2)-3/2}^{-1}\cdot\cdots\cdot q_{\theta(n+1/2)+1/2}^{-1} & (\theta(n-1/2)>\theta(n+1/2)).
\end{cases}
\]
For $\sigma$ and $\ulam$ and a symmetric function $f$, let $f(\mathbf{q}^{w(\sigma,\ulam\sss;\sss\varepsilon_\sigma,\varepsilon_{\ulam})})$ denote the power series obtained by substituting $\mathbf{q}^{w(h)}$ for $h\in\Zh$ such that $\sigma(h)=\varepsilon_\sigma$ and $\ulam(h)=\varepsilon_{\ulam}$.
For a symmetric function $f=f(p_1,p_2,\ldots)$, let $f^*$ be the symmetric function given by $f^*(p_1,p_2,\ldots)=f(p_1^{-1},p_2^{-1},\ldots)$. 
For a subset $S\subset \Zh$, let $f(S\sss;\sss\mathbf{q})$ denote the symmetric function given by substituting $\{q_h\mid h\in S\}$ for $f$. 
We set 
\[
S_{\sigma,\ulam}^{\varepsilon,\iota}:=\{h\in\Zh\mid\sigma(h)=\varepsilon,\, \ulam(h)=\iota\}.
\]
\end{NB}%
\begin{defn}
\begin{align*}
&\mathcal{Z}^{\zeta>0}_{\sigma,\unu,\ulam}(q_0,\ldots,q_{L-1})
:=
\mathbf{q}(\type)\cdot \\
&\Biggl(\ \sum_{{\nu_-\supseteq \nu_1\supseteq \nu_2\subseteq\nu_3\subseteq\nu_+}}
s_{\nu_-\backslash\nu_1}\bigl(S^{+,-}_{\sigma,\ulam}\sss;\sss\mathbf{q}\sss\bigr)\cdot 
\tenchi s_{\nu_1\backslash\nu_2}^*\bigl(S^{-,-}_{\sigma,\ulam}\sss;\sss\mathbf{q}\sss\bigr)\cdot 
s_{\nu_3\backslash\nu_2}\bigl(S^{+,+}_{\sigma,\ulam}\sss;\sss\mathbf{q}\sss\bigr)\cdot 
\tenchi s_{\nu_+\backslash\nu_3}^*(S^{-,+}_{\sigma,\ulam}\sss;\sss\mathbf{q}\sss\bigr)
\Biggr).
\end{align*}
\end{defn}
Combining Corollary \ref{cor_ncdt_crystal}, Lemma \ref{lem_commutator}, Proposition \ref{prop_crystal_vo}, Lemma \ref{lem_3.16} and Lemma \ref{lem_3.17}, we get the following explicit formula:
\begin{thm}\label{thm_main}
If we put $t=q_0\cdot\cdots\cdot q_{L-1}$, then we have
\begin{NB}
\begin{align*}
&\mathcal{Z}^{\mathrm{ncDT}}_{\ztpair}(q_0,\ldots,q_{L-1})\\
&=
\mathbf{q}^{-\mathbf{v}(\Vmin)}\cdot 
M(1,t)^L\cdot
\prod_{j}s_{\lambda^{(j)}}(t^{-\rho})\cdot
\Biggl(\,
\prod_{
\alpha\in \Delta^{\mathrm{re},+},
\theta(\alpha)<0,
}
(1+\sigma(\alpha)\mathbf{q}^\alpha)^{\sigma(\alpha)\alpha_0}
\Biggr)
\\
&\quad\cdot\Biggl(\ \sum_{{\nu_-\supseteq \nu_1\supseteq \nu_2\subseteq\nu_3\subseteq\nu_+}}
s_{\nu_-\backslash\nu_1}\bigl(S^{+,-}_{\sigma,\ulam}\sss;\sss\mathbf{q}\sss\bigr)\cdot 
\tenchi s_{\nu_1\backslash\nu_2}^*\bigl(S^{-,-}_{\sigma,\ulam}\sss;\sss\mathbf{q}\sss\bigr)\cdot 
s_{\nu_3\backslash\nu_2}\bigl(S^{+,+}_{\sigma,\ulam}\sss;\sss\mathbf{q}\sss\bigr)\cdot 
\tenchi s_{\nu_+\backslash\nu_3}^*(S^{-,+}_{\sigma,\ulam}\sss;\sss\mathbf{q}\sss\bigr)
\Biggr).
\end{align*}
\end{NB}%
\begin{align*}
&\mathcal{Z}^{\mathrm{E-ncDT}}_{\ztpair}(q_0,\ldots,q_{L-1})=\\
&
M(1,t)^L\cdot
\prod_{j}s_{\lambda^{(j)}}(t^{-\rho})\cdot
\Biggl(\,
\prod_{
\alpha\in \Delta^{\mathrm{re},+},
\theta(\alpha)<0,
}
(1+\sigma(\alpha)\mathbf{q}^\alpha)^{\sigma(\alpha)\alpha_0}
\Biggr)\cdot
\mathcal{Z}^{\zeta>0}_{\sigma,\unu,\ulam}(q_0,\ldots,q_{L-1}).
\end{align*}
\end{thm}
\begin{cor}\label{cor_WC}
The normalized generating function
\[
\mathbf{q}(\type)\cdot\mathcal{Z}^{\mathrm{E-ncDT}}_{\type}(q_0,\ldots,q_{L-1})/\mathcal{Z}^{\mathrm{E-ncDT}}_{\sigma,\theta,\uemp,\uemp}(q_0,\ldots,q_{L-1})
\]
does not depend on $\theta$.
\end{cor}

\section{Open Donaldson-Thomas invariants and topological vertex}\label{sec_5}
\subsection{Open Donaldson-Thomas invariants}\label{subsec_openDT}
\begin{NB}
Take $\beta\in H_2(Y_\sigma)$ and $n\in\Z$.
\begin{defn}
A $(\sss\unu,\ulam\sss)$-pair of type $(\beta,n)$ is a pair $(F,S)$ of a coherent sheaf $F\in\cohcy$ and a morphism $S\colon \I_{(\underline{\nu},\underline{\lambda})}\to F$ such that $c_2(F)=\beta$, $\chi(F)=n$ and $S$ is surjective.

A family of $(\sss\unu,\ulam\sss)$-pair of type $(\beta,n)$ over $B$ is a pair $(\F,\mcaS)$ of a coherent sheaf $\F$ on $Y_\sigma\times B$ and a morphism $\mcaS\colon \I_{(\underline{\nu},\underline{\lambda})}\times \mathcal{O}_{Y_\sigma}\to \F$ such that $\F$ is flat over $B$ and 
$\F|_b$ is a $(\sss\unu,\ulam\sss)$-pair of type $(\beta,n)$ for any closed point $b\in B$.

Two family $(\F,\mcaS)$ and $(\F',\mcaS')$ are said to be isomorphic is there exists an isomorphism between $\F$ and $\F'$ compatible with $\mcaS$ and $\mcaS'$.
\end{defn}
The moduli functor of $(\sss\unu,\ulam\sss)$-pair on $B$ of type $(\beta,n)$ is the functor given by 
\[
\begin{array}{ccc}
\mathbf{Var}^{\mathrm{op}} & \to & \mathbf{Sets}\\
B & \to & \{\textup {A family of $(\sss\unu,\ulam\sss)$-pair on $B$ of type $(\beta,n)$}\}/\sim
\end{array}
\]
\begin{thm}
The moduli functor of $(\sss\unu,\ulam\sss)$-pair on $B$ of type $(\beta,n)$ is represented by a quasi-projective variety $\mathfrak{M}^{\mathrm{DT}}(\sss\underline{\nu},\underline{\lambda}\,;\sss\beta,n)$.
\end{thm}
\end{NB}
Take $\beta\in H_2(Y_\sigma,\Z)$ and $n\in\Z$.
Note that $H_2(Y_\sigma,\Z)$ has the natural basis $\{[C_i]\}_{i=1,\ldots,L-1}$, where $C_i\simeq \CP^1$ is an irreducible component of the fiber $f^{-1}(0)$ of the contraction $f\colon Y_\sigma\to X$. 
The derived equivalence induces the following isomorphism:
\[
\begin{array}{cccc}
\psi\ \colon & K_{\mathrm{num}}(\modfa) & \overset{\sim}{\longrightarrow} & H_2(Y_\sigma,\Z)\oplus\Z\,,\\
& {[}S_i{]} \ (i\neq 0) & \longmapsto & [C_i]\in H_2(Y_\sigma,\Z),\\
& {[}S_0{]}+ \cdots +{[}S_{L-1}{]} & \longmapsto & 1\in \Z.
\end{array}
\]
\begin{defn}
A $(\sss\unu,\ulam\sss)$-pair of type $(\beta,n)$ is a pair $(F,s)$ of a coherent sheaf $F\in\cohcy$ and a morphism $s\colon \I_{\underline{\nu},\underline{\lambda}}\to F$ such that $c_2(F)=\beta$, $\chi(F)=n$ and $s$ is surjective.

Two $(\sss\unu,\ulam\sss)$-pairs $(F,s)$ and $(F',s')$ are said to be equivalent if there exists a isomorphism between $F$ and $F'$ compatible with $s$ and $s'$.
\end{defn}
Recall that in \cite{3tcy} we construct a tilting vector bundle $\mathcal{P}:=\mathcal{O}_{Y_\sigma}\oplus \bigoplus_i L_i$ on $Y_\sigma$ following \cite{vandenbergh-3d}.
In particular, we have
\begin{equation}\label{eq_tilting_bundle}
(2-L)[\mathcal{O}_{Y_\sigma}]+\sum_i[L_i]=[\mathcal{L}]\in K_0(\cohy)
\end{equation}
where $\mathcal{L}$ is an ample line bundle on $Y_\sigma$. 

The functor $\R\Hom(\mathcal{L},-)$ gives an equivalence between $\dcohy$ (resp. $\dcohcy$) and $\dmoda$ (resp. $\dmodfa$), which restricts to an equivalence between ${}^{-1}\mathrm{Per}(Y_\sigma/X)$ and $\moda$. 
Here ${}^{-1}\mathrm{Per}(Y_\sigma/X)$ is the full subcategory of $\dcohy$ consisting of elements $E$ satisfying the following conditions:
\begin{itemize}
\item $H^i_{\cohy}(E)=0$ unless $i=0,-1$,
\item $\R^1f_*(H^0_{\cohy}(E))=0$ and $\R^0f_*(H_{\cohy}^{-1}(E))=0$,
\item $\Hom(H^{0}_{\cohy}(E),C)=0$ for any sheaf $C$ on $Y$ satisfying $\R f_*(C)=0$.
\end{itemize}
Take $\zeta_{\mathrm{cyc}}^{\circ}=(1-L,1,1,\ldots,1)$. 
Note that $(\zeta^\circ,T')$ is not on an intersection of two walls for any $T'\in\R$ (see \S \ref{subsec_15}).
\begin{lem}\label{lem_phase}
For an element $E\in \cohcy\cap \modfa$ we have $\phi_{Z_{\zeta_{\mathrm{cyc}}^{\circ}}}(E)\leq 1/2$ and for an element $E[1]\in \cohcy[1]\cap \modfa$ we have $\phi_{Z_{\zeta_{\mathrm{cyc}}^{\circ}}}(E[1])>1/2$.
\end{lem}
\begin{proof}
By \eqref{eq_tilting_bundle}, for an element $E\in\cohcy$ we have
\[
\zeta_{\mathrm{cyc}}^{\circ}\cdot \dimv E=h^0(E,Y_\sigma)-h^0(E\otimes \mathcal{L},Y_\sigma)\leq 0.
\]
and the equality holds if and only if $E$ is $0$-dimensional. 
Since any $0$-dimensional sheaf is in $\modfa$, the claim follows.
\end{proof}
\begin{lem}\label{lem_inclusion}
\[
\mathcal{D}^{\zeta^\circ_{\mathrm{cyc}}}_{\mathrm{fin}}[1/2,0)\subset\cohcy, \quad\mathcal{D}^{\zeta^\circ_{\mathrm{cyc}}}_{\mathrm{fin}}[1,1/2)\subset\cohcy[1].
\]
\end{lem}
\begin{proof}
Let $E\in\dmodfa$ be a $Z_{\zeta_{\mathrm{cyc}}^{\circ}}$-semistable object with $1/2\geq \phi_{Z_{\zeta_{\mathrm{cyc}}^{\circ}}}(E)>0$. 
By the canonical exact sequence 
\[
0\to H^{-1}_{\cohy}(E)[1]\to E \to H^0_{\cohy}(E)\to 0.
\]
Since $H^{-1}_{\cohy}(E)[1]\in \cohcy[1]\cap \modfa$, we have
\[
\phi_{Z_{\zeta_{\mathrm{cyc}}^{\circ}}}(H^{-1}_{\cohy}(E)[1])>1/2
\]
by Lemma \ref{lem_phase}. 
Then we can see $H^{-1}_{\cohy}(E)[1]=0$ and so $E\in \cohcy$.
We can show the right inclusion in the same way. 
\end{proof}
\begin{prop}\label{prop_44}
\[
\mathcal{D}^{\zeta^\circ_{\mathrm{cyc}}}_{\mathrm{fin}}[1/2,-1/2)\simeq\cohcy.
\]
\end{prop}
\begin{proof}
Every object $F\in \mathcal{D}^{\zeta^\circ_{\mathrm{cyc}}}_{\mathrm{fin}}[1/2,-1/2)$ fits into a short exact sequence
\[
0\to E\to F\to G\to 0
\]
for some pair of objects $E\in \mathcal{D}^{\zeta^\circ_{\mathrm{cyc}}}_{\mathrm{fin}}[1/2,0)$ and $G\in \mathcal{D}^{\zeta^\circ_{\mathrm{cyc}}}_{\mathrm{fin}}[0,-1/2)$. 
By Lemma \ref{lem_inclusion}, we have $E,G\in \cohcy$ and so $F\in \cohcy$.
Since both $\mathcal{D}^{\zeta^\circ_{\mathrm{cyc}}}_{\mathrm{fin}}[1/2,-1/2)$ and $\cohcy$ are cores of t-structures, the inclusion is equivalence and the claim follows.
\end{proof}
\begin{lem}\label{lem_G}
Let $G$ be an $A_\sigma$-module.
Suppose that $\mathrm{Hom}(X,G)=0$ for any finite dimensional $A_\sigma$-module $X$. 
Then we have $G\in \mathrm{Coh}(Y_\sigma)$.
\end{lem}
\begin{proof}
Recall that $\mathrm{mod}(A_\sigma)$ coincides with $\perv$ (Theorem \ref{thm_perv}). 
Thus we have the following exact sequence in $\mathrm{mod}(A_\sigma)$: 
\[
0\to H^{-1}_{\mathrm{coh}(Y_\sigma)}(G)\to G\to H^{0}_{\mathrm{coh}(Y_\sigma)}(G)\to 0.
\]
Since the restriction of an element in $\perv$ to the smooth locus of $X$ is a sheaf, the support of $H^{-1}_{\mathrm{coh}(Y_\sigma)}(G)$ is compact. 
Thus, as an $A_\sigma$-module, $H^{-1}_{\mathrm{coh}(Y_\sigma)}(G)$ is finite dimensional.

By the assumption we have $H^{-1}_{\mathrm{coh}(Y_\sigma)}(G)=0$. Hence the claim follows.
\end{proof}
\begin{lem}
\[
\mathcal{D}^{\zeta^\circ_{\mathrm{cyc}}}[1/2,0)^\bot
\subset\cohy.
\]
\end{lem}
\begin{proof}
In the proof of Proposition \ref{prop_torsion_pair}, we show that any element $F\in \mathcal{D}^{\zeta^\circ_{\mathrm{cyc}}}[1/2,0)^\bot$ fits into an exact sequence
\[
0\to E\to F\to G\to 0
\]
where $E\in \mathcal{D}^{\zeta^\circ_{\mathrm{cyc}}}_{\mathrm{fin}}[1/2,0)$ and $\mathrm{Hom}(X,G)=0$ for any finite dimensional $A_\sigma$-module $X$.
By Lemma \ref{lem_inclusion} and Lemma \ref{lem_G}, we have $F\in \mathrm{Coh}(Y_\sigma)$.
\end{proof}
\begin{NB}
Since we have
\[
\mathcal{D}^{\zeta^\circ_{\mathrm{cyc}}}[1/2,-1/2)\simeq\cohy
\]
we can show the following in the same way:
\begin{lem}
\[
\mathcal{D}^{\zeta^\circ_{\mathrm{cyc}}}_{\mathrm{fin}}[1,1/2)^{\bot}\subset\cohcy.
\]
\end{lem}
\begin{proof}
In the proof of Proposition \S \ref{prop_torsion_pair}, we show that for any $F\in \mathcal{D}^{\zeta^\circ_{\mathrm{cyc}}}_{\mathrm{fin}}[1,1/2)^{\bot}$

\end{proof}
\begin{prop}\label{prop_47}
\[
\mathcal{D}^{\zeta^\circ_{\mathrm{cyc}}}_{\mathrm{fin}}[1/2,-1/2)=\cohcy, \quad
\mathcal{D}^{\zeta^\circ_{\mathrm{cyc}}}[1/2,-1/2)=\cohy.
\]
\end{prop}
\end{NB}
\begin{prop}\label{prop_47}
\[
\mathcal{D}^{\zeta^\circ_{\mathrm{cyc}}}[1/2,-1/2)\simeq\cohy.
\]
\end{prop}
\begin{proof}
Using the previous lemma, we can prove the claim in the same way as Proposition \ref{prop_44}.
\end{proof}
\begin{thm}
There is a coarse moduli scheme $\mathfrak{M}^{\mathrm{DT}}(\sss\underline{\nu},\underline{\lambda}\,;\sss\beta,n)$ parameterizing equivalence classes of $(\sss\unu,\ulam\sss)$-pairs $(F,s)$ of type $(\beta,n)$.
\end{thm}
\begin{proof}
By the Noetherian property, we can take sufficiently small $T>0$ such that $\I_{\underline{\nu},\underline{\lambda}}\in \catAzeta$ for $\zeta=\zeta^\circ_{\mathrm{cyc}}+T\eta$. 
Moreover we can assume that for any positive root $\alpha\leq \psi^{-1}(\beta,n)$ and for any $T>T'>0$, $\zeta':=\zeta^\circ_{\mathrm{cyc}}+T\eta$ is not on the wall $W_\alpha$. 
Then, using Propositions \ref{prop_44} and \ref{prop_47} and by the same argument as in \cite[\S 2]{nagao-nakajima}, we can verify that giving a $(\sss\unu,\ulam\sss)$-pair is equivalent to giving a $(\ztpair)$-pair. Hence the claim follows from Theorem \ref{thm_moduli}.
\end{proof}
\begin{NB}
\begin{defn}
For $\beta\in H_2(Y_\sigma)$ and $n\in\Z$, we define
\[
\mathfrak{M}^{\mathrm{E-DT}}(\sss\underline{\nu},\underline{\lambda}\,;\sss\beta,n):=
\{
\I_{(\underline{\nu},\underline{\lambda})}\surj F\mid F\in \mathrm{Coh}^{\mathrm{cpt}}(Y),\ c_2(F)=\beta,\ \chi(F)=n\}.
\]
\end{defn}
\end{NB}
\begin{rem}
An alternative construction for $\mathfrak{M}^{\mathrm{DT}}(\sss\underline{\nu},\underline{\lambda}\,;\sss\beta,n)$ is the following: 
first, take a compactification $\overline{Y}$ of $Y$ and let $\overline{\mathcal{I}_{\underline{\nu},\underline{\lambda}}}$ be the ideal sheaf on $\overline{Y}$. 
Then we can get the moduli scheme as an open subscheme of the quot scheme for $\overline{\mathcal{I}_{\underline{\nu},\underline{\lambda}}}$. 
\end{rem}
\begin{cor}
Take sufficiently small $T>0$ and put $\zeta=\zeta^\circ_{\mathrm{cyc}}+T\eta$, then we have
\[
\mathfrak{M}^{\mathrm{DT}}(\sss\underline{\nu},\underline{\lambda}\,;\sss\beta,n)\simeq \MMM.
\]
\end{cor}
\begin{defn}
We define the generating function
\[
\mathcal{Z}^{\mathrm{E-DT}}_{\sigma,\underline{\nu},\underline{\lambda}}(q_1,\ldots,q_{L-1},t):=
\sum_{n,\beta} e\Bigl(\mathfrak{M}^{\mathrm{DT}}(\sss\underline{\nu},\underline{\lambda}\,;\sss\beta,n)\Bigr)\cdot \mathbf{q}^{\beta}t^n,
\]
where $\mathbf{q}^{\beta}:=\prod(q_i)^{\beta_i}$ for $\beta=\sum\beta_i[C_i]$.
\end{defn}

\subsection{Topological vertex via vertex operators}\label{subsec_tv_via_vo}
Let $\vec{\nu}=(\nu^{(1)},\ldots,\nu^{(L-1)})$ be an ($L-1$)-tuple of Young diagrams and 
$\vec{\Lambda}=(\Lambda^{(1/2)},\ldots,\Lambda^{(L-1/2)})$ be an $L$-tuple of $3$-dimensional Young diagrams such that $\Lambda^{(j)}$ is
\begin{itemize}
\item of type $(\lambda^{(j)},\nu^{(j+1/2)},{}^{\mathrm{t}}\nu^{(j-1/2)})$ if $\sigma(j)=+$,
\item of type $(\lambda^{(j)},{}^{\mathrm{t}}\nu^{(j-1/2)},\nu^{(j+1/2)})$ if $\sigma(j)=-$,\label{item_weight}
\end{itemize}
where we put $\nu^{(0)}:=\nu_-$ and $\nu^{(L)}:=\nu_+$.

For a ($L-1$)-tuple of Young diagrams $\vec{\nu}$, we define the weight 
\[
w(\vec{\nu}):=\prod_i\prod_{(x,y)\in \nu^{(i)}}
\begin{cases}
q_i\cdot t^{2x+1} & \sigma(i+1/2)=\sigma(i-1/2)=+,\\
q_i\cdot t^{2y+1} & \sigma(i+1/2)=\sigma(i-1/2)=-,\\
q_i\cdot t^{x+y+1} & \sigma(i+1/2)\neq \sigma(i-1/2),
\end{cases}
\]
and for a $3$-dimensional Young diagram $\Lambda$ of type $(\lambda_x,\lambda_y,\lambda_z)$ we define the weight $w(\Lambda)$ by 
\[
w(\Lambda):=t^{|\Lambda\backslash \Lambda_{\mathrm{min}}|}.
\]
For a datum $(\vec{\mu},\vec{\Lambda})$ as above, we define the weight $w(\vec{\mu},\vec{\Lambda})$ by
\[
w(\vec{\mu},\vec{\Lambda}):=w(\vec{\nu})\cdot \prod w(\Lambda^{(j)}).
\]

The $T:=(\C^*)^3$-action on $Y_\sigma$ induces a $T$-action on $\mathfrak{M}^{\mathrm{DT}}(\sss\underline{\nu},\underline{\lambda}\sss;\sss\beta,n)$.
The following claim is given in \cite{mnop}.
\begin{prop}
The set $\mathfrak{M}^{\mathrm{DT}}(\sss\underline{\nu},\underline{\lambda}\sss;\sss\beta,n)^T$ of $T$-fixed points is isolated and parametrized by the data $(\vec{\mu},\vec{\Lambda})$ as above with weight $\mathbf{q}^{\beta}\cdot t^n$.
\end{prop}
\begin{defn}
We define the generating function
\[
\mathcal{Z}_{\unu,\ulam}^{\mathrm{TV}}(q_1,\ldots,q_{L-1},t):=\sum_{(\vec{\mu},\vec{\Lambda})} w(\vec{\mu},\vec{\Lambda}).
\]
\end{defn}
\begin{cor}
If we put $t=q_0\cdot\cdots\cdot q_{L-1}$, then we have
\[
\mathcal{Z}_{\sigma,\unu,\ulam}^{\mathrm{E-DT}}(q_0,\ldots,q_{L-1})
=\mathcal{Z}_{\sigma,\unu,\ulam}^{\mathrm{TV}}(q_1,\ldots,q_{L-1},t)
\]
\end{cor}
We set 
\[
H^{j}_\pm:=\{h\in\Zh\mid \pi(h)=j,\ \ulam(h)=\pm\}.
\]
The following claim directly follows from the argument in \cite{ORV}:
\begin{prop}
\begin{align*}
&\mathcal{Z}^{\mathrm{TV}}_{\type}(q_1,\ldots,q_{L-1},t)|_{t=\mathbf{q}^\delta}=\mathbf{q}(\type)\cdot \\
&\quad\Bigl\langle\, \nu_-\,\Big|\,
\prod_{h\in H^{1/2}_-}\Gamma^{\sigma(1/2)}_{-}({q}_h)\cdot
\prod_{h\in H^{1/2}_+}\Gamma^{\sigma(1/2)}_{+}({q}_h)\cdot
\prod_{h\in H^{3/2}_-}\Gamma^{\sigma(3/2)}_{-}({q}_h)\cdot\\
&\quad\cdots\cdot
\prod_{h\in H^{L-1/2}_-}\Gamma^{\sigma(L-1/2)}_{-}({q}_h)\cdot
\prod_{h\in H^{L-1/2}_-}\Gamma^{\sigma(L-1/2)}_{+}({q}_h)
\,\Big|\,\nu_+\,\Bigr\rangle.
\end{align*}
\end{prop}
We set
\[
\Delta^{\mathrm{re},+}_-:=\{\alpha_{[h,h']}\in \Delta^{\mathrm{re},+}\mid \pi(h)>\pi(h')\}.
\]
We can compute the generating function in the same way as Theorem \ref{thm_main}.
\begin{thm}
\begin{align*}
&\mathcal{Z}^{\mathrm{TV}}_{\type}(q_1,\ldots,q_{L-1},t)\\
&=
\mathbf{q}(\type)\cdot 
M(1,t)^L\cdot
\prod_{j}s_{\lambda^{(j)}}(t^{-\rho})\cdot
\Biggl(\,
\prod_{
\alpha\in \Delta^{\mathrm{re},+}_-
}
(1+\sigma(\alpha)\mathbf{q}^\alpha)^{\sigma(\alpha)\alpha_0}
\Biggr)
\\
&\quad\cdot\Biggl(\ \sum_{{\nu_-\supseteq \nu_1\supseteq \nu_2\subseteq\nu_3\subseteq\nu_+}}
s_{\nu_-\backslash\nu_1}\bigl(S^{+,-}_{\sigma,\ulam}\sss;\sss\mathbf{q}\sss\bigr)\cdot 
\tenchi s_{\nu_1\backslash\nu_2}^*\bigl(S^{-,-}_{\sigma,\ulam}\sss;\sss\mathbf{q}\sss\bigr)\cdot 
s_{\nu_3\backslash\nu_2}\bigl(S^{+,+}_{\sigma,\ulam}\sss;\sss\mathbf{q}\sss\bigr)\cdot 
\tenchi s_{\nu_+\backslash\nu_3}^*(S^{-,+}_{\sigma,\ulam}\sss;\sss\mathbf{q}\sss\bigr)
\Biggr).
\end{align*}
\end{thm}
\begin{rem}
In fact, we do not use non-commutative Donaldson-Thomas theory in the proof of this theorem.
\end{rem}

\section{Construction of moduli spaces}\label{sec_4}
\subsection{Moduli space via a framed quiver with relations}\label{subsec_proof}
Since $A_\sigma^\zeta$ is Noetherian (see Lemma \ref{lem_noether}), we can take a presentation
\[
\oplus P_i^{\oplus b_i}\overset{J}{\to} \oplus P_i^{\oplus a_i}\to \III\to 0
\]
of the finitely generate $A_\sigma^\zeta$-module $\III$. 
Given a presentation, we define the new quiver with relation $\bar{A}^\zeta_\sigma(\unu,\ulam):=\C\bar{Q}/\mathcal{J}$ as follows: 
\begin{itemize}
\item the set of vertices of $\bar{Q}$ is given by $I\sqcup *$, 
\item the set of arrows of $\bar{Q}$ is given by adding $\iota_a$ ($1\leq a\leq a_i$) from $*$ to $i$ for each $i$ to the set of arrows of $Q$, and
\item the ideal $\mathcal{J}$ is generated by the relations of the original algebra $A_\sigma^\zeta$ and the elements of the following form:
\[
\sum \gamma_a(B)\cdot \iota_a
\]
for $B\in \oplus P_i^{\oplus b_i}$ and $J(B)=\sum \gamma_a(B)\cdot e_a$, where $e_a\in \oplus P_i^{\oplus a_i}$ is the idempotent in the direct summand corresponding to the index $a$.
\end{itemize}
Let $P_*$ (resp. $S_*$) be the projective (resp. simple) $\bar{A}^\zeta_\sigma$-module associated with the vertex $*$. 
The following claim is clear from the construction:
\begin{lem}
The kernel of the natural projection $P_*\to S_*$, as an $A^\zeta_\sigma$-module, is isomorphic to $\III$.
\end{lem}
According to this lemma, we have the natural isomorphism 
\begin{align*}
\mathrm{Ext}^1_{\bar{A}^\zeta_\sigma(\unu,\ulam)}(S_*,V)&\simeq 
\mathrm{Hom}_{\bar{A}^\zeta_\sigma(\unu,\ulam)}(\mathrm{Ker}(P_*\to S_*),V)\\
&\simeq
\mathrm{Hom}_{{A}^\zeta_\sigma}(\III,V).
\end{align*}

Moreover, $s\in \mathrm{Hom}_{{A}^\zeta_\sigma}(\III,V)$ is surjective if and only if the $\bar{A}^\zeta_\sigma(\unu,\ulam)$-module $V_s$ is generated by $(V_s)_*$, where $V_s$ is given by the extension
\[
0\to V\to V_s\to S_*\to 0
\]
corresponding to $s$.

Take 
\begin{align*}
{\theta}_{\mathrm{cyc}}&\in (K_{\mathrm{num}}(\mathrm{mod}\,\bar{A}^\zeta_\sigma(\unu,\ulam))\otimes \R)^*\\
&\simeq (K_{\mathrm{num}}(\mathrm{mod}\,A^\zeta_\sigma)\otimes \R)^*\oplus \R
\end{align*}
such that 
\[
{\theta}_{\mathrm{cyc}}\cdot (\mathbf{v},1)=0,\quad ({\theta}_{\mathrm{cyc}})_i>0\ (i\in I).
\]
Then, the surjectivity condition above is equivalent to ${\theta}_{\mathrm{cyc}}$-stability.  
Hence we can construct the moduli space $\MMM$ as King's moduli space of ${\theta}_{\mathrm{cyc}}$-stable $\bar{A}^\zeta_\sigma(\unu,\ulam)$-modules with dimension vector $=(\mathbf{v},1)\in K_{\mathrm{num}}(\mathrm{mod}\,\bar{A}^\zeta_\sigma(\unu,\ulam))\simeq \Z^I\oplus\Z$.
\begin{rem}
For $\zeta\in (K_{\mathrm{num}}(\mathrm{mod}\,A_\sigma)\otimes \R)^*$, we can define the moduli space 
$\bar{\mathfrak{M}}^{{\zeta}}_{{A}_\sigma}(\mathbf{v})$ of ${\zeta}$-semistable framed ${A}_\sigma$-modules as in \cite{nagao-nakajima}.
In \textup{\cite{3tcy}}, it is shown that 
\[
\bar{\mathfrak{M}}^{{\zeta}}_{{A}_\sigma}(\mathbf{v})\simeq 
\mathfrak{M}_{\bar{A}^\zeta_\sigma(\uemp,\uemp)}(\mathbf{v},1)\Bigl(\simeq 
\mf{M}^{\mathrm{ncDT}}(\zeta,\uemp,\uemp\,;\vv)\Bigr).
\]
\end{rem}

\subsection{Moduli space via a framed quiver with a potential}\label{subsec_potential}
Assume $\nu=\uemp$.
We put 
\begin{align*}
\bigvee(\theta,\ulam)&:=\{n\in\Z\mid \ulam\circ\theta(n-1/2)=-,\,\ulam\circ\theta(n+1/2)=+\}\\
&\,=\{n(1/2),\ldots,n(K+1/2)\},\\
\bigwedge(\theta,\ulam)&:=\{n\in\Z\mid \ulam\circ\theta(n-1/2)=+,\,\ulam\circ\theta(n+1/2)=-\}\\
&\,=\{n(1),\ldots,n(K)\}, 
\end{align*}
where $n(1/2)<n(1)<\cdots<n(K+1/2)$.
We consider a newer quiver $\hat{Q}=\hat{Q}^\zeta_\sigma(\uemp,\ulam)$ obtained from $Q^\zeta_\sigma$ by adding 
\begin{itemize}
\item an arrow $\iota_a$ from $*$ to $\pi(n(a))$ for $a=1/2,\ldots,K+1/2$, and
\item an arrow $\tau_b$ from $\pi(n(b))$ to $*$ for $b=1,\ldots,K$.
\end{itemize}
We define a new potential $\hat{w}=\hat{w}^\zeta_\sigma(\uemp,\ulam)$ for the new quiver $\hat{Q}$ by
\[
\hat{w}:=w+\sum_{b=1}^{K}\Bigl(\tau_{b}\circ h_{[n(b-1/2),n(b)]}\circ\iota_{b-1/2}-\tau_b\circ h_{[n(b+1/2),n(b)]}\circ\iota_{b+1/2}\Bigr),
\]
where $w$ is the original potential and 
\begin{align*}
h_{[n(b-1/2),n(b)]}&:=h_{\pi(n(b))-1/2}^+\circ\cdots\circ h_{\pi(n(b-1/2))+1/2}^+,\\
h_{[n(b+1/2),n(b)]}&:=h_{\pi(n(b))+1/2}^-\circ\cdots\circ h_{\pi(n(b+1/2))-1/2}^-.
\end{align*}
Let $\hat{A}^\zeta_\sigma(\uemp,\ulam)$ be the Jacobi algebra of the $(\hat{Q},\hat{w})$.

We take ${\theta}_{\mathrm{cyc}}$ as in the previous subsection. 
Since the relations of $\hat{A}^\zeta_\sigma(\uemp,\ulam)$ is obtained by the derivations of the potential, the moduli space $\mathfrak{M}^{\theta_{\mathrm{cyc}}}_{\hat{A}^\zeta_\sigma(\uemp,\ulam)}(\mathbf{v},1)$ is the critical locus of a regular function and admits a symmetric obstruction theory.

We can show the following claim in the same way as in \cite[\S Proposition 4.7]{nagao-nakajima}. 
\begin{prop}
\[
\mathfrak{M}^{\theta_{\mathrm{cyc}}}_{\bar{A}^\zeta_\sigma(\uemp,\ulam)}(\mathbf{v},1)\simeq 
\mathfrak{M}^{\theta_{\mathrm{cyc}}}_{\hat{A}^\zeta_\sigma(\uemp,\ulam)}(\mathbf{v},1).
\]
\end{prop}
Combined with the result in the previous subsection, we see that $\mf{M}^{\mathrm{ncDT}}(\zeta,\underline{\emptyset},\underline{\lambda}\,;\vv)$ admits a symmetric obstruction theory.

\begin{NB}
For $\zeta\in (K_{\mathrm{num}}(\mathrm{mod}\,A_\sigma(\unu,\ulam))\otimes \R)^*$, we can define the moduli space $\mathfrak{M}^{\bar{\zeta}}_{\hat{A}_\sigma(\uemp,\ulam)}(\mathbf{v},1)$ as in the previous subsection.
\begin{defn}
\[
\hat{\mathcal{Z}}^{\mathrm{ncDT}}_{\sigma,\zeta,\uemp,\ulam}(q_0,\ldots,q_{L-1}):=
\sum_\mathbf{v} e\Bigl(\mathfrak{M}^{\bar{\zeta}}_{\hat{A}_\sigma(\uemp,\ulam)}(\mathbf{v},1)\Bigr)\cdot \mathbf{q}^{\mathbf{v}}
\]
\end{defn}
Due to \cite[Theorem 7.5]{joyce-song}. we can apply the argument in \cite{3tcy} (or the result in \cite{joyce-4}) to get the wall-crossing formula:
\begin{prop}
The normalized generating function
\[
\hat{\mathcal{Z}}^{\mathrm{ncDT}}_{\sigma,\theta,\uemp,\ulam}(q_0,\ldots,q_{L-1})/\mathcal{Z}^{\mathrm{ncDT}}_{\sigma,\theta,\uemp,\uemp}(q_0,\ldots,q_{L-1})
\]
does not depend on $\theta$.
\end{prop}
\begin{cor}
\[
\hat{\mathcal{Z}}^{\mathrm{ncDT}}_{\sigma,\zeta,\uemp,\ulam}(q_0,\ldots,q_{L-1})
=\mathcal{Z}^{\mathrm{ncDT}}_{\sigma,\zeta,\uemp,\ulam}(q_0,\ldots,q_{L-1})
\]
\end{cor}
\end{NB}

\section{Remarks and appendices}
\subsection{Weighted Euler characteristic}\label{subsec_w}
Let $\nu\colon \mf{M}^{\mathrm{ncDT}}(\zeta,\underline{\emptyset},\underline{\lambda}\,;\vv)\to \Z$ be the constructible function defined in \cite{behrend-dt} (Behrend function).
We define the generating function
\[
\mathcal{Z}^{\mathrm{ncDT}}_{\sigma,\zeta\sss;\sss\underline{\emptyset},\underline{\lambda}}(q_0,\ldots,q_{L-1}):=
\sum_\mathbf{v}\Bigl(\sum_{n\in\Z}n\cdot\chi(\nu^{-1}(n))\Bigr)
\cdot (\mathbf{q}_\theta)^{\mathbf{v}}
\]
where $(\mathbf{q}_\theta)^{\mathbf{v}}:=\prod(\mathbf{q}^{\alpha(\theta,i)})^{v_i}$ as in \S \ref{subsec_14}.
The Behrend function is defined for any scheme over $\mathbb{C}$. 
In \cite{behrend-dt}, Behrend showed that if a proper scheme has a symmetric obstruction theory then the virtual counting, which is defined by integrating the constant function $1$ over the virtual fundamental cycle, coincides with the weighted Euler characteristic weighted by the Behrend function as above.
Based on this result, he proposed to define the virtual counting for a non-proper variety with a symmetric obstruction theory as the weighted Euler characteristic.

We can apply Behrend-Fantechi's result \cite[Theorem3.4]{behrend-fantechi} to compute the weighted Euler characteristic by torus localization.
Using the ``Kozsul like'' complex (\cite[Equation (140)]{joyce-song}) 
and the same argument as \cite[Theorem 7.1]{ncdt-brane}, we can compute the parity of the dimension of Zariski tangent space at a torus fixed point, which determines the contribution of the torus fixed point. 
As a result we get
\[
\mathcal{Z}^{\mathrm{ncDT}}_{\sigma,\zeta\sss;\sss\underline{\emptyset},\underline{\lambda}}(q_0,\ldots,q_{L-1})=
\mathcal{Z}^{\mathrm{E-ncDT}}_{\sigma,\zeta\sss;\sss\underline{\emptyset},\underline{\lambda}}(p_0,\ldots,p_{L-1})
\]
under the variable change given by 
\[
\mathbf{p}^{\alpha(\theta,i)}=(-1)^{\hat{Q}^\zeta_\sigma(\uemp,\ulam)_{i\to i}+\hat{Q}^\zeta_\sigma(\uemp,\ulam)_{i\to *}+\hat{Q}^\zeta_\sigma(\uemp,\ulam)_{*\to i}}\mathbf{q}^{\alpha(\theta,i)}
\]
where $\hat{Q}^\zeta_\sigma(\uemp,\ulam)_{i\to j}$ is the number of arrows in the quiver $\hat{Q}^\zeta_\sigma(\uemp,\ulam)$ from the vertex $i$ to the vertex $j$.

\subsection{Pandharipande-Thomas invariants and wall-crossing}\label{subsec_NCPT}
As we mentioned at the end of \S \ref{subsec_13}, we have worked on the area $\{\zeta\mid \zeta\cdot \delta<0\}$ up to now.
In this subsection, we make some comments on the area $\{\zeta\mid \zeta\cdot \delta>0\}$, i.e. $0<t<1/2$.
We have a natural bijection between $\Theta$ and the set of chambers in the area $\{\zeta\mid \zeta\cdot \delta>0\}$ as well.
An element $\zeta_\theta$ in the chamber $C^\theta$ corresponding to $\theta\in\Theta$ satisfies the following condition:
\[
\alpha_{[h,h']}\cdot \zeta_\theta>0
\iff
\theta(h)<\theta(h')
\]
for any $h<h'$.

Note that $\D{t}{0}$ is a left admissible full subcategory of $\dmoda$ for $0<t<1$.
Let $\tenchi\D{t+1}{t}$ denote the core of t-structure given from $\moda$ by tilting with respect to $\D{t}{0}$\footnote{This is different from $\Dinf{t}{t-1}[1]$ given in \S \ref{subsec_stability_and_tilting}.} and we put
\[
\tenchi\hspace{-2pt}\catAinfzeta:=\tenchi\D{t+1}{t}[-1].
\]
(see Definition \ref{defn_catA}).
As in Proposition \ref{prop_t_str} we have
\[
\tenchi\hspace{-2pt}\catAinfzeta\simeq \mathrm{mod}_{\mathrm{fin}}A_\sigma^\zeta.
\]
We put
\[
\tenchi\III:=H^0_{\tenchi\hspace{-2pt}\catAinfzeta}(\II)
\]
(see \eqref{eq_III}) and 
let $\mathfrak{M}^{\mathrm{ncPT}}(\zeta,\underline{\nu},\underline{\lambda}\,;\vv)$ denote the moduli space of quotient objects $V$ of $\tenchi\III$ in $\tenchi\hspace{-2pt}\catAinfzeta$ with $[V]=\vv$:
(see Definition \ref{defn_pair} and Theorem \ref{thm_moduli}).
\begin{ex}\label{ex_finite_pyramid}
Assume that $L_+=L_-=1$ $\unu=\uemp$ and $\ulam=\uemp$. 
Take $\zeta^{m.+}\in \Z^2$ such that 
\[
\zeta^{m,+}_0<\zeta^{m,+}_1,\quad
m\zeta^{m,+}_0+(m-1)\zeta^{m,+}_1<0,\quad
(m+1)\zeta^{m,+}_0+m\zeta^{m,+}_1>0.
\]
as in \textup{\cite[\S 4.1]{nagao-nakajima}}. 
Then we have
\[
P^{\zeta^{m,+}}_{\uemp.\uemp}=\mathcal{O}_C^{(m,m-1,\ldots,1)}
\]
where $\mathcal{O}_C^{(m,m-1,\ldots,1)}$ is given by the exact sequence
\[
0\to \mathcal{I}(\uemp,\uemp;(m,m-1,\ldots,1))\to \mathcal{O}_{Y_\sigma}\to \mathcal{O}_C^{(m,m-1,\ldots,1)}\to 0.
\]
As an $A_{\sigma\circ\theta}$-module, $\mathcal{O}_C^{(m,m-1,\ldots,1)}$ is given by ``the finite type pyramid with length $m$'' (\textup{\cite{chuang_pan}}).
Hence the moduli $\mathfrak{M}^{\mathrm{ncPT}}(\zeta^{m,+},\uemp,\uemp\,;\vv)$ coincides with the one we denoted by $\mathfrak{M}^{A_m^+}_{\zeta_{\mathrm{cyclic}}}(\vv)$ in \textup{\cite[\S 4]{nagao-nakajima}}.
\end{ex}
Take the limit of $\theta$, the following moduli appears:
\[
\bigl\{s\colon\II\to F\mid F:\text{ pure of dimension $1$},\ \dim\ker(s)=1\bigr\}.
\]
This is a generalization of the moduli of stable pairs in the sense of Pandharipande-Thomas (\cite{pt1}).
Thus let us call the Euler characteristic of the moduli space $\mathfrak{M}^{\mathrm{ncPT}}(\zeta,\underline{\nu},\underline{\lambda}\,;\vv)$ as the open non-commutative Pandharipande-Thomas (open ncPT in short) invariant temporarily.

On the other hand, we define the finite type transition of Young diagrams (and hence the finite type crystal model) of type $(\type)$ by replacing $\theta$ by $(-1)\circ\theta$, where $(-1)\colon\Zh\to\Zh$ is the multiplication of $(-1)$. 
The generating function is described by the operator
\[
\prod^{(-1)\circ\theta}_{h\in\Zh}\Gamma^{\sigma(h)}_{\lambda(h)}(q_h)
\]
and we can compute it in the same way as \S \ref{subsec_computation}\footnote{A similar crystal melting model and computation are given in \cite{Sulkowski}.}.

The author expects that as an $A_\sigma^\zeta$-module $\tenchi\III$ is given by the grand state finite type crystal and hence we can compute the generating function of open ncPT invariants explicitly.
We can check this is true for some concrete examples (see Example \ref{ex_finite_pyramid}). 
If this is true in general, 
\begin{itemize}
\item we get open version of DT-PT correspondence in our setting, and 
\item we can realize the normalized generating function appearing in Corollary \ref{cor_WC} as the generating function of Euler characteristics of the moduli spaces 
\[
\mathfrak{M}^{\mathrm{ncPT}}(\zeta^{\mathrm{triv}},\underline{\nu},\underline{\lambda}\,;\vv)
\]
for $\zeta^{\mathrm{triv}}\in\R^I$ such that $(\zeta^{\mathrm{triv}})_i>0$ for any $i$.
\end{itemize}

\begin{NB}
\subsection{Invariants via cluster category}\label{sec_6}
In this subsection, we use the standard notations in cluster category theory, which are different from the ones in the other part of this paper. 
Let $(Q,w)$ be a quiver with a potential. 
In \cite{amiot}, it is assumed that the Jacobi algebra $J(Q,w)$ is finite dimensional. 
Let $\Gamma=\Gamma(Q,w)$ be the Ginzburg's dg algebra. We denote by $\mathcal{D}\Gamma$ the derived category of dg $\Gamma$-modules and by $\mathcal{D}^b\Gamma$ its full subcategory formed by the dg $\Gamma$-modules whose cohomology is of finite total dimension. 
Let $\mathrm{per}\Gamma$ denote the category of perfect dg $\Gamma$-modules.
Note that $\mathcal{D}^b\Gamma$ is a full subcategories of $\mathrm{per}\Gamma$ since $\Gamma$ is homologically smooth.
\begin{defn}[\protect{\cite[\S 3.3]{amiot}}]
The {\it cluster category} $\mathcal{C}=\mathcal{C}_{(Q,w)}$ associated to $(Q,w)$ is the quotient of triangulated categories $\mathrm{per}\Gamma/\mathcal{D}^b\Gamma$.
\end{defn}
From our point of view, the image of the projection to $\mathcal{C}$ determines the asymptotic behaviour of an element in $\mathrm{per}\Gamma$, and elements on a fiber of the projection is what we want to count.
\begin{defn}[\protect{\cite[\S 2.2]{amiot}}]
The {\it fundamental domain} $\mathcal{F}$ is the following full subcategory of $\mathrm{per}\Gamma$\textup{;}
\[
\mathcal{F}:=\mathcal{D}_{\leq 0}\cap{}^\bot\mathcal{D}_{\leq -2}\cap \mathrm{per}\Gamma.
\]
\end{defn}
Note that $\mathcal{C}$ is intrinsic to the triangulated category $\mathrm{per}\Gamma$, but $\mathcal{F}$ is not.  
\begin{prop}[\protect{\cite[Proposition 2.9]{amiot}}]
The projection functor $\pi\colon \mathrm{per}\Gamma\to \mathcal{C}$ induces a $\C$-linear equivalence $\Psi$ between $\mathcal{F}$ and $\mathcal{C}$. 
\end{prop}
The presentation of $\III$ we took in \S \ref{subsec_proof} is $\Psi^{-1}(\II)$.

Assume that $H^i(\Gamma)=0$ for $i<0$, i.e. the Jacobi algebra $J$ is $3$-Calabi-Yau.
For an element $I\in\mathcal{C}$, take 
\[
\Psi^{-1}=(P_1\overset{f}{\to} P_0)
\]
where $P_1$ and $P_0$ is projective $J$-modules. 
We define the moduli space $\mathfrak{M}^{J}(I,\mathbf{v})$ by 
\[
\mathfrak{M}^{J}(I,\mathbf{v}):=
\{P_1\overset{f}{\to} P_0\to V\to 0\mid [V]=\mathbf{v}\}
\]
where $\mathbf{v}\in \Z^I$. Here an isomorphism of objects is given by the following diagram:
\[
\xymatrix{
P_1 \ar[r]^f \ar[d]^{\mathrm{id}}  & P_0 \ar[r] \ar[d]^{\mathrm{id}}  & V \ar[d]^{\simeq}\\
P_1 \ar[r]^f & P_0 \ar[r] & V'. \\
}
\]
Note that this does not depend on the choice of $f$ and we can construct $\mathfrak{M}^{J}(I,\mathbf{v})$ in the same way as \S \ref{subsec_proof}.
\begin{defn}
We define the generating function $\mathcal{Z}^J(I;\mathbf{q})$ by 
\[
\mathcal{Z}^J(I;\mathbf{q}):=\sum_{\mathbf{v}}e(\mathfrak{M}^{J}(I,\mathbf{v}))\cdot \mathbf{q}^{\mathbf{v}}.
\]
\end{defn}
As we mentioned at the bottom of the previous subsection, the invariants we studied in this paper is a special case of the invariants above. 
The following question is a generalization of the result of this paper:
\begin{ques}
Given a derived equivalence between $\Gamma$ and $\Gamma'$, study the relation between $\mathcal{Z}^J(I;\mathbf{q})$ and $\mathcal{Z}^{J'}(I;\mathbf{q})$.
\end{ques}
A typical example of such derived equivalences is the one given by a mutation of a quiver with a potential (\cite{dong-keller,keller-completion}).
\end{NB}

\subsection{Proof of Proposition \ref{prop_sec_6}}\label{sec_7}
\begin{lem}[\protect{\cite[Lemma 3.2]{young-mckay}}]
For two Young diagrams $\mu$ and $\mu'$, $\mu\pg\mu'$ if and only if $\tenchi\mu_c-\tenchi\mu'_c=0$ or $1$ for any $c\in\Z_{\geq 0}$.
\end{lem}
Note that $\Vmin(n)$ is determined by $\Vmin(n-1)$ and $\Vmin(n+1)$.
\begin{ex}
Assume that $\sigma\circ\theta(n-1/2)=\sigma\circ\theta(n+1/2)=\ulam\circ\theta(n-1/2)=\ulam\circ\theta(n+1/2)=+$. 
Then we have $\tenchi\V(n-1)_c-\tenchi\V(n+1)_c=0$, $1$ or $2$ for any transition $\V$ and
\[
\tenchi\Vmin(n-1)_c-\tenchi\Vmin(n)_c=
\begin{cases}
0 & \tenchi\Vmin(n-1)_c-\tenchi\Vmin(n+1)_c=0,\\
1 & \tenchi\Vmin(n-1)_c-\tenchi\Vmin(n+1)_c=1 \text{ or } 2.\\
\end{cases}
\]
\begin{NB}
\noindent\textup{(2)}
Assume that $\sigma\circ\theta(n-1/2)=\sigma\circ\theta(n+1/2)=\ulam\circ\theta(n+1/2)=+$, $\ulam\circ\theta(n-1/2)=-$. 
Then we have $\tenchi\V(n-1)_c-\tenchi\V(n+1)_c=-1$, $0$ or $1$ for any transition $\V$ and
\[
\tenchi\Vmin(n-1)_c-\tenchi\Vmin(n)_c=
\begin{cases}
-1 & \tenchi\Vmin(n-1)_c-\tenchi\Vmin(n+1)_c=-1,\\
0 & \tenchi\Vmin(n-1)_c-\tenchi\Vmin(n+1)_c=0 \text{ or } 1.\\
\end{cases}
\]

\noindent\textup{(3)}
Assume that $\sigma\circ\theta(n-1/2)=\sigma\circ\theta(n+1/2)=\ulam\circ\theta(n+1/2)=+$, $\ulam\circ\theta(n-1/2)=-$. 
Then we have $\tenchi\V(n-1)_c-\tenchi\V(n+1)_c=-1$, $0$ or $1$ for any transition $\V$ and
\[
\tenchi\Vmin(n-1)_c-\tenchi\Vmin(n)_c=
\begin{cases}
-1 & \tenchi\Vmin(n-1)_c-\tenchi\Vmin(n+1)_c=-1 \text{ or } 0,\\
0 & \tenchi\Vmin(n-1)_c-\tenchi\Vmin(n+1)_c=1.
\end{cases}
\]
\end{NB}%
\end{ex}
For a transition $\V$ of Young diagram, an {\it addable $i$-node} for $\V$ is an element $(n,x,y)$ such that $\pi(n)=i$, $(x,y)\neq \V(n)$ and 
\[
\V'(m):=\begin{cases}
\V(m) & m\neq n,\\
\V(n)\sqcup (x,y) & m=n,
\end{cases}
\]
gives a transition. 
Let $\V^{[i]}$ denote the transition given by adding all addable $i$-node for $\V$. 
\begin{ex}
\begin{NB}
\textup{(1)} 
Unless $\ulam\circ\theta(n-1/2)=-$ and $\ulam\circ\theta(n+1/2)=+$, there is no addable $i$-node in $\Vmin(n)$.
\end{NB}%
Assume that $\sigma\circ\theta(n-1/2)=\sigma\circ\theta(n+1/2)=\ulam\circ\theta(n-1/2)=\ulam\circ\theta(n+1/2)=+$. 
Then $\V^{[i]}(n\pm 1)=\V(n\pm 1)$ and
\[
\tenchi\Vmin(n-1)_c-\tenchi\Vmin(n)_c=
\begin{cases}
0 & \tenchi\Vmin(n-1)_c-\tenchi\Vmin(n+1)_c=0 \text{ or } 1,\\
1 & \tenchi\Vmin(n-1)_c-\tenchi\Vmin(n+1)_c=1.\\
\end{cases}
\]
\begin{NB}
\noindent\textup{(2)}
Assume that $\sigma\circ\theta(n-1/2)=\sigma\circ\theta(n+1/2)=\ulam\circ\theta(n+1/2)=+$, $\ulam\circ\theta(n-1/2)=-$. 
Then $\V^{[i]}(n\pm 1)=\V(n\pm 1)$ and
\[
\tenchi\V^{[i]}(n-1)_c-\tenchi\V^{[i]}(n)_c=
\begin{cases}
-1 & \tenchi\V^{[i]}(n-1)_c-\tenchi\V^{[i]}(n+1)_c=-1 \text{ or } 0,\\
0 & \tenchi\V^{[i]}(n-1)_c-\tenchi\V^{[i]}(n+1)_c=1 .
\end{cases}
\]
\end{NB}%
\end{ex}
For a transition $\V$, $M(\V^{[i]})$ does not have the simple module $S_i$ as its quotient. 
Note that for an $A_\sigma$-module $M$ without the simple module $S_i$ as its quotient, $M$ is identified with an $A_{\sigma\circ\theta\circ\theta_{i}}$-module $\mu_iM$ in $\dmoda$, and 
\[
(\mu_iM)_{i'}=
\begin{cases}
M_{k} & k\neq i,\\
\mathrm{ker}\bigl(M_{i-1}\oplus M_{i+1}\to M_i\bigr) & k=i.
\end{cases}
\]

Let $\Vmin':=\Vmin^{\sigma,\theta\circ\theta_i,\unu,\ulam}$ be the minimal transition of type $(\sigma,\theta\circ\theta_i,\unu,\ulam)$.
\begin{prop}\label{prop_mutation}
As $A_{\sigma\circ\theta\circ\theta_{i}}$-modules, $\mu_iM(\Vmin^{[i]})$ is isomorphic to $M(\Vmin')$. 
\end{prop}
\begin{proof}
Note that $\Vmin(n)=\Vmin'(n)$ if $\pi(n)\neq i$.
We will define isomorphisms 
\[
\mathrm{ker}\bigl(M(\Vmin)_{n-1}\oplus M(\Vmin)_{n+1}\to M(\Vmin)_n\bigr)\overset{\sim}{\longrightarrow}
M(\Vmin')_n.
\]
for $n$ such that $\pi(n)=i$. For example, assume that $\sigma\circ\theta(n-1/2)=\sigma\circ\theta(n+1/2)=\ulam\circ\theta(n-1/2)=\ulam\circ\theta(n+1/2)=+$. 
Note that $\Vmin=\Vmin'$ in this case.
For $n$ such that $\pi(n)=i$, 
\[
\mathrm{ker}\bigl(M(\Vmin^{[i]})_{n-1}\oplus M(\Vmin^{[i]})_{n+1}\to M(\Vmin^{[i]})_n\bigr)
\]
is spanned by the following elements: 
\[
\begin{array}{lr}
-p(n+1,x,\tenchi\Vmin(n)_x) \quad &  \tenchi\Vmin(n-1)_x-\tenchi\Vmin(n+1)_x=0,\\
p(n-1,x,y)-p(n+1,x,y-1) & \tenchi\Vmin(n-1)_x-\tenchi\Vmin(n+1)_x=0,\\
& \text{and}\  y=2,\ldots,\tenchi\Vmin(n)_x,\\
p(n-1,x,y)-p(n+1,x,y-1) &  \tenchi\Vmin(n-1)_x-\tenchi\Vmin(n+1)_x=1,\\
& \text{and}\  y=2,\ldots,\tenchi\Vmin(n)_x+1,\\
p(n-1,x,\tenchi\Vmin(n)_x+1) & \tenchi\Vmin(n-1)_x-\tenchi\Vmin(n+1)_x=2,\\
p(n-1,x,y)-p(n+1,x,y-1) & \tenchi\Vmin(n-1)_x-\tenchi\Vmin(n+1)_x=2,\\
& \text{and}\ y=2,\ldots,\tenchi\Vmin(n)_x.
\end{array}
\]
The isomorphism is given by mapping one of elements above involving $p(n-1,x,y)$ or $p(n+1,x,y-1)$ to $p'(n,x,y)$. 
\begin{NB}
\begin{description}
\item{if $\tenchi\Vmin(n-1)_x-\tenchi\Vmin(n+1)_x=0$}: $p(n+1,x,\tenchi\Vmin(n)_x)$ and $p(n-1,x,y)-p(n+1,x,y-1)$ ($y=2,\ldots,\tenchi\Vmin(n)_x)$).
\item[if $\tenchi\Vmin(n-1)_x-\tenchi\Vmin(n+1)_x=1$]: $p(n-1,x,y)-p(n+1,x,y-1)$ ($y=2,\ldots,\tenchi\Vmin(n)_x+1)$).
\item[if $\tenchi\Vmin(n-1)_x-\tenchi\Vmin(n+1)_x=2$]: $p(n-1,x,\tenchi\Vmin(n)_x+1)$ and $p(n-1,x,y)-p(n+1,x,y-1)$ ($y=2,\ldots,\tenchi\Vmin(n)_x)$).
\end{description}
\end{NB}%
We can verify this isomorphism respects the actions of $A_{\sigma\circ\theta\circ\theta_{i}}$.
\end{proof}

\noindent{\bf Proof of Proposition \ref{prop_sec_6}}

\noindent (Step 1) 
In the case $\theta=\mathrm{id}$, $\unu=\uemp$ and $\ulam=\uemp$, we have $\Vmin^{\sigma,\mathrm{id},\uemp,\uemp}(n)=\emptyset$ for any $n$ and $M(\Vmin^{\sigma,\mathrm{id},\uemp,\uemp})\simeq P_0\simeq \mathcal{O}_{Y_\sigma}$.
By Proposition \ref{prop_mutation}, we have $M(\Vmin^{\sigma,\mathrm{id},\uemp,\uemp})\simeq \mathcal{O}_{Y_\sigma}\simeq P_{\uemp,\uemp}^\theta$.

\smallskip

\noindent (Step 2)
The main result in \cite[\S 6.4]{3tcy} is the following:
\begin{quote}
Given $\beta$ and $n$, we can take sufficiently small $T>0$ such that giving a crystal of type $(\type)$ is equivalent to giving a data $(\vec{\nu},\vec{\Lambda})$ as in \S \ref{subsec_tv_via_vo}.
\end{quote}
More precisely, we have $L$ intervals $C_{1/2},\ldots,C_{L-1/2}$ disjoint with each other such that in the interval $C_j$ 
\[
\bigsqcup_{n\in C_j}\V(n)
\]
is identified with the $3$-dimensional Young diagram $\Lambda^{(j)}$.
Here, a box in $\Lambda^{(j)}$ is identified with an $L$-tuple elements from successive Young diagrams $\V(n),\ldots,\V(n+L-1)$ for some $n$, which is isomorphic to the skyscraper sheaf of the $j$-th fixed point in $Y_\sigma$ as an $A_{\sigma\circ\theta}$-module.

\smallskip

\noindent (Step 3)
Take sufficiently small $T>0$, such that 
\[
0\to \mathcal{O}_{Z_{\underline{\nu},\underline{\lambda}}} \to \mathcal{O}_{Y_\sigma}\to \II\to 0
\]
is an exact sequence is $\catAzeta$ for $\zeta=\zeta^\circ_{\mathrm{cyc}}+T\eta$.
Note that the closed subscheme $Z_{\underline{\nu},\underline{\lambda}}\subset Y_\sigma$ is decomposed into the disjoint union of closed subschemes
\[
Z_{\underline{\nu},\underline{\lambda}}^{(j)}\subset U^{(j)}\subset Y_\sigma
\]
where $U^{(j)}$ is the toric coordinate locus around the $j$-th fixed point. 

Hence, what we have to show is 
\[
\bigoplus_{n\in C_j,\,(x,y)\in \V(n)}\C\cdot(n,x,y)
\]
is isomorphic to $\mathcal{O}_{\underline{\nu},\underline{\lambda}}^{(j)}:=\mathcal{O}_{Z_{\underline{\nu},\underline{\lambda}}^{(j)}}$ as an $A_{\sigma\circ\theta}$-module.

\smallskip

\noindent (Step 4)
In \cite[Proposition 3.10]{3tcy}, it is shown that the derived equivalence $\Phi$ between $Y_\sigma$ and $A_{\sigma\circ\theta}$ is given by a tilting vector bundle on $Y_\sigma$ which is a direct sum of line bundles. 
Hence $\Phi(\mathcal{O}_{\underline{\nu},\underline{\lambda}}^{(j)})_i$ is isomorphic to $\mathcal{O}_{\underline{\nu},\underline{\lambda}}^{(j)}$, and hence to\[
\bigoplus_{\begin{subarray}{c}n\in C_j,\,\pi(n)=i,\\(x,y)\in \V(n)\end{subarray}}\C\cdot(n,x,y)
\]
as $\mathcal{O}_X$-modules for any $i$. 

Moreover, in \cite[Proposition 3.10]{3tcy} we described the map between line bundles corresponding to the arrow $h^\pm_{k}$ in the quiver, which induces an endomorphisms on $\mathcal{O}_{\underline{\nu},\underline{\lambda}}^{(j)}$. 
We can check this endomorphism coincides with
\[
\bigoplus_{\begin{subarray}{c}n\in C_j,\,\pi(n)=k\mp1/2,\\(x,y)\in \V(n)\end{subarray}}\C\cdot(n,x,y)\to
\bigoplus_{\begin{subarray}{c}n\in C_j,\,\pi(n)=k\pm1/2,\\(x,y)\in \V(n)\end{subarray}}\C\cdot(n,x,y)
\]
given by the Definition \ref{defn_cryctal_rep}. 
Hence Proposition \ref{prop_sec_6} follows.
\vspace{-6mm}\begin{flushright}{$\square$}\end{flushright}
\begin{ex}
In Figure \ref{fig_1}, \ref{fig_3} and \ref{fig_4}, we provide some examples which may help the reader to understand the proof. 
All the examples are in the case of $L_+=L_-=1$, i.e. the conifold case.
In Figure \ref{fig_1} we provide the figure of a part of the grand state crystal in the case of $\unu=\uemp$ and $\ulam=\uemp$.
In Figure \ref{fig_3} and \ref{fig_4} we provide the figures of grand state and 
\[
\bigsqcup_{n\in C_{3/2}}\V(n).
\]
\end{ex}

\begin{figure}[htbp]
  \centering
  \input{fig_1.tpc}
  \caption{$\unu=\uemp$ and $\ulam=\uemp$.}
\label{fig_1}
\end{figure}

\begin{figure}[htbp]
  \centering
  \input{fig_3.tpc}
  \caption{$\unu=\uemp$ and $\ulam=(\emptyset,\square)$}
\label{fig_3}
\end{figure}

\begin{figure}[htbp]
  \centering
  \input{fig_4.tpc}
  \caption{$\unu=(\emptyset,\square)$ and $\ulam=(\emptyset,\square\hspace{-1.5pt}\square)$}
\label{fig_4}
\end{figure}

\bibliographystyle{amsalpha}
\bibliography{bib-ver5}

\end{document}